\documentclass[a4paper, intlimits, 10pt]{amsart}
\calclayout
\usepackage{enumerate}
\usepackage{color}
\usepackage{amssymb}
\usepackage{extarrows}
\usepackage{graphicx}
\usepackage{breakcites}
\usepackage{dsfont}
\definecolor{dark green}{rgb}{0.09, 0.45, 0.27}
\usepackage{hyperref}
\usepackage{tcolorbox}

\newtheorem{theorem}{Theorem}[section]
\newtheorem{proposition}[theorem]{Proposition}
\newtheorem{lemma}[theorem]{Lemma}
\newtheorem{corollary}[theorem]{Corollary}
\newtheorem{definition}[theorem]{Definition}
\newtheorem{remark}[theorem]{Remark}

\newtheorem{assumption}[theorem]{Assumption}

\newtheorem{mtheorem}[theorem]{Metatheorem}
\newcommand{\R}{\mathbb{R}}

\newcommand{\N}{\mathds{N}}

\newcommand{\E}{\mathbb{E}}

\renewcommand{\H}{\mathcal{H}}

\renewcommand{\P}{\mathds{P}}

\newcommand{\Pd}{\mathrm{Proj}^d}

\newcommand{\Pts}{\mathrm{Proj}_{\theta\#}}

\makeatletter
\newtheorem*{rep@theorem}{\rep@title}
\newcommand{\newreptheorem}[2]{%
\newenvironment{rep#1}[1]{%
 \def\rep@title{#2 \ref{##1}}%
 \begin{rep@theorem}}%
 {\end{rep@theorem}}}
\makeatother

\newreptheorem{theorem}{Theorem}
\newreptheorem{corollary}{Corollary}

\usepackage{tikz}
\usetikzlibrary{decorations.markings}
\usetikzlibrary{shapes.geometric}

\pgfdeclarelayer{edgelayer}
\pgfdeclarelayer{nodelayer}
\pgfsetlayers{edgelayer,nodelayer,main}

\tikzstyle{none}=[inner sep=0pt]
\tikzset{new/.style={thick}}

\definecolor{cccccc}{rgb}{0.8,0.8,0.8}
\definecolor{azure}{rgb}{0.0, 0.5, 1.0}
\definecolor{cqcqcq}{rgb}{0.75,0.75,0.75}
\definecolor{ao}{rgb}{0.0, 0.5, 0.0}
\definecolor{amber}{rgb}{1.0, 0.49, 0.0}
\definecolor{babyblue}{rgb}{0.54, 0.81, 0.94}
\definecolor{darkpastelgreen}{rgb}{0.01, 0.75, 0.24}
\definecolor{darkspringgreen}{rgb}{0.09, 0.45, 0.27}

\begin{document}
\title[Max-sliced Wasserstein concentration on RKHS]{Max-sliced Wasserstein concentration and uniform ratio bounds of empirical measures on RKHS}

\date{\today}

\author{Ruiyu Han}
\address{
Ruiyu Han\newline
Department of Mathematical Sciences\newline 
Carnegie Mellon University\newline 
5000 Forbes Avenue\newline
Pittsburgh, PA 15213, USA}
\email{ruiyuh@andrew.cmu.edu}

\author{Cynthia Rush}
\address{
Cynthia Rush\newline
Department of Mathematics\newline
Columbia University\newline
1255 Amsterdam Avenue\\
New York, NY 10027, USA}
\email{cgr2130@columbia.edu}

\author{Johannes Wiesel}
\address{
Johannes Wiesel \newline
Department of Mathematical Sciences\newline 
Carnegie Mellon University\newline 
5000 Forbes Avenue\newline
Pittsburgh, PA 15213, USA}
\email{wiesel@cmu.edu}


\begin{abstract}
Optimal transport and the Wasserstein distance $\mathcal{W}_p$ have recently seen a number of applications in the fields of statistics, machine learning, data science, and the physical sciences. These applications are however severely restricted by the curse of dimensionality, meaning that the number of data points needed to estimate these problems accurately increases exponentially in the dimension. To alleviate this problem, a number of variants of $\mathcal{W}_p$ have been introduced. We focus here on one of these variants, namely the max-sliced Wasserstein metric $\overline{\mathcal{W}}_p$. This metric reduces the high-dimensional minimization problem given by $\mathcal{W}_p$ to a maximum of one-dimensional measurements in an effort to overcome the curse of dimensionality. In this note we derive concentration results and upper bounds on the expectation of $\overline{\mathcal{W}}_p$ between the true and empirical measure on unbounded reproducing kernel Hilbert spaces. We show that, under quite generic assumptions, probability measures concentrate uniformly fast in one-dimensional subspaces, at (nearly) parametric rates. Our results rely on an improvement of currently known bounds for $\overline{\mathcal{W}}_p$ in the finite-dimensional case.
\end{abstract}

\maketitle

\section{Introduction}

In recent years, the theory of optimal transportation (OT) of probability measures \cite{villani2009optimal, rachev1998mass} has seen applications in the fields of statistics, machine learning, data science, and the physical sciences, see \cite{rubner2000earth, courty2017joint, chernozhukov2017monge, arjovsky2017wasserstein, alvarez2018gromov, gramfort2015fast} to name just a few. A well known specific instance of an OT problem is given by the Wasserstein metric
\begin{align}\label{eq:wass}
\mathcal{W}_p(\mu,\nu)^p := \inf_{\pi\in \Pi(\mu,\nu)} \int \|x-y\|_{\mathcal{H}}^p\,\pi(dx,dy). 
\end{align}
Here $\mu$ and $\nu$ are two probability measures on some underlying Hilbert space $(\mathcal{H},\|\cdot\|_{\mathcal{H}})$ and $\Pi(\mu,\nu)$ is the set of joint distributions of $\mu$ and $\nu$ on the product space $\mathcal{H}\times \mathcal{H}.$ From a topological point of view, two of the most attractive features of $\mathcal{W}_p$ are its compatibility with the weak topology and its isometric embedding of the underlying space $(\mathcal{H},\|\cdot\|_{\mathcal{H}})$ through Dirac measures.
Unfortunately however, its statistical applications are severely restricted by the \emph{curse of dimensionality}, meaning that typically $\mathcal{W}_p(\mu_n,\mu) \approx n^{-1/d}$ for $d>2p$, where $\mu_n$ is the empirical measure of i.i.d.\ samples from $\mu$ and $\mathcal{H}$ has dimension $d.$ In other words, its rate of concentration degrades drastically as the dimension of the problem grows. To alleviate this, a number of modifications to the Wasserstein metric have been proposed, see e.g., \cite{goldfeld2020convergence, goldfeld2020gaussian, kolouri2015radon, deshpande2018generative}. In this work, we focus on one specific modification, namely the \emph{max-sliced Wasserstein metric} $\overline{\mathcal{W}}_p$ (see \eqref{eq:MSW} below for a formal definition). The max-sliced Wasserstein metric reduces the high-dimensional minimization problem \eqref{eq:wass} to a maximum of one-dimensional measurements in an effort to overcome the curse of dimensionality. Indeed, it turns out that $\overline{\mathcal{W}}_p(\mu_n,\mu)$ concentrates at parametric rates under various conditions on the underlying measures in the input space, which is usually given by the Euclidean space $\mathcal{H}=\mathbb{R}^d$. 

In this work, we extend these results to unbounded Hilbert spaces $\mathcal{H}$: we derive dimension-free upper bounds for $\overline{\mathcal{W}}_p(\mu_n,\mu)$ when $\mu$ is a probability measure on $\mathcal{H}$.  In other words, under quite general conditions, probability measures concentrate uniformly at fast rates in one-dimensional subspaces of a Hilbert space. Even in the finite-dimensional case, our results improve on currently known bounds on the expectation and concentration of $\overline{\mathcal{W}}_p(\mu_n,\mu)$.

Concentration of probability measures on (reproducing kernel) Hilbert spaces finds relevant and timely applications in statistics and machine learning, as kernels are often a powerful tool for measuring similarity between data points when the data has complicated structure, like graphical data. Moreover, many machine learning problems are formulated in feature space,  instead of the original input space, due to the ability of the Hilbert space construction to capture nonlinear structures in the data. It is therefore essential to establish concentration results for $\overline{\mathcal{W}}_p(\mu_n,\mu)$ \emph{after} transforming the data from the original input space into an infinite-dimensional Hilbert space through a kernel function to evaluate the performance of such methods. 

In the remainder of this introduction, we give a more formal overview of our main results. For this, we recall that $\mathcal{H}$ denotes a Hilbert space with scalar product $\langle \cdot, \cdot\rangle_\mathcal{H}$ and norm $\|\cdot\|_\mathcal{H}$.
Given a probability measure $\mu$ on $\mathcal{H}$ and a unit vector $\theta\in \H$, we define the pushforward measure $\mu_\theta$ via $$\mu_\theta(A)=\mu(\{x\in \mathcal{H}: \langle x, \theta\rangle_{\mathcal{H}} \in A\}),\qquad  A\subseteq \R \text{ Borel}.$$
The main object of this paper is the \emph{max-sliced Wasserstein distance of order $p$}, with $p\in [1,\infty)$, for measures $\mu, \nu \in \mathcal{H}$, given by
\begin{align}
\label{eq:MSW}
\overline{\mathcal{W}}_p(\mu,\nu):=\sup\limits_{\|\theta\|_{\H}=1}\mathcal{W}_p (\mu_\theta, \nu_\theta).
\end{align}
Here, $\mathcal{W}_p$ denotes the Wasserstein distance of order $p$ between the one-dimensional distributions $\mu_\theta$ and $\nu_\theta$, i.e.,
\begin{align*}
\mathcal{W}_p (\mu_\theta, \nu_\theta)^p =\inf_{\pi \in \Pi(\mu_\theta, \nu_\theta)} \int |x-y|^p\,\pi(dx,dy). 
\end{align*}
In short, $\overline{\mathcal{W}}_p(\mu,\nu)$ is the 
maximum Wasserstein distance between the pushforward measures $\mu_\theta$ and $\nu_\theta$ of the one-dimensional projections $x\mapsto \langle x, \theta\rangle$. The max-sliced Wasserstein distance goes back to the early works \cite{deshpande2019max,paty2019subspace} and it is well known that it metrizes weak convergence. In fact, if $\mathcal{H}$ is finite dimensional, then the topology induced by $\overline{\mathcal{W}}_p$ is the $p$-Wasserstein topology \cite{bayraktar2021strong}. In infinite dimensional spaces, $\overline{\mathcal{W}}_p$ still metrizes the weak topology; we give an overview of these results in Section \ref{sec:topology}.

In this note we are primarily interested in the statistical properties of $\overline{\mathcal{W}}_p$: writing $X_1, \dots, X_n$ for an i.i.d.\ sample from $\mu$ and defining the empirical measure $\mu_n:= \frac{1}{n} \sum_{j=1}^n \delta_{X_j},$ we aim to answer the following question:

\begin{tcolorbox}
What are (tight) upper bounds for  $\overline{\mathcal{W}}_p(\mu_n,\mu)$? Or more generally: how well does $\mu_n$ concentrate uniformly over the one-dimensional subspaces of $\mathcal{H}$? 
\end{tcolorbox}

Our interest in this question derives from the following observation: while the $p$-Wasserstein distance $\mathcal{W}_p$ on a finite-dimensional space $\mathcal{H}$ typically exhibits rates of order $\E[\mathcal{W}_p(\mu_n,\mu)^p]\approx n^{-p/d}$ for $d>2p$ (i.e.\ it suffers from the curse of dimensionality), this is \emph{not} true for its max-sliced counterpart. Indeed, if $\mu$ is subgaussian, then it has been observed by \cite{niles2022estimation} that 
$\E[\overline{\mathcal{W}}_p(\mu_n,\mu)]\lesssim \sqrt{d} n^{-1/(2p)}$ for general $p\ge 1$. This analysis was later extended by \cite{lin2021projection} to the case where $\mu$ satisfies a Bernstein tail condition or Poincar\'e inequality. For the special case that $\mu$ is isotropic, \cite{bartl2022structure} give sharp concentration bounds on the rate of convergence for $p=2$. Very recently, this analysis has been extended to non-isotropic measures and general Banach spaces in the case $p=1$ \cite{march}. To the best of our knowledge, this is the first paper establishing rates for $\overline{\mathcal{W}}_p(\mu_n,\mu)$ in infinite dimensional unbounded spaces and general $p\in [1,\infty).$

Notably, all of the works mentioned above establish that $\overline{\mathcal{W}}_p(\mu_n,\mu)$ does not suffer from the curse of dimensionality. However, even in the finite dimensional case, it remains unclear if  the assumptions of \cite{niles2022estimation, lin2021projection} can be weakened for general $p\ge 1$. Our first result shows that this is indeed the case: if $\mathcal{H}$ is finite dimensional, then 
\begin{align}\label{eq:intro_rate}
\E[\overline{\mathcal{W}}_p(\mu_n,\mu)^p]\lesssim \sqrt{\frac{d}{n}},
\end{align}
as soon as $M_s(\mu):=\int \|x\|^{s}_\mathcal{H}\,\mu(dx) <\infty$ for some $s>2p,$ see Theorem \ref{thm:exp_fin} in Section \ref{sec:main}. Here $\lesssim$ hides logarithmic factors of $n$ and a constant depending on $p,s$ and $M_s(\mu)$ only. In consequence, this result covers arbitrary heavy-tailed measures $\mu$ with finite $s$-moment; no subgaussian tails, Bernstein tail conditions or Poincar\'e inequalities need to be checked. In additional to the result in \eqref{eq:intro_rate}, we also obtain concentration inequalities for $\overline{\mathcal{W}}_p(\mu_n,\mu)$, see Theorem \ref{thm:conc_fin}. For ease of exposition we however focus on rates for the expectation in this introduction, and refer to a more complete discussion of our results in Section \ref{sec:main}. Interestingly, our method of proof for these results significantly deviates from the works \cite{niles2022estimation,lin2021projection} mentioned above. Instead of deriving rates through $\epsilon$-net arguments and sufficiently light tails of $\mu$, we instead use bounds on the uniform ratio
\begin{align}\label{eq:uniform_ratio_intro}
\sup_{(\mathbf{\theta},t)\in \H \times {\mathbb R}}\frac{|F_{\mathbf{\theta}}(t)-F_{\mathbf{\theta},n}(t)|}{\sqrt{F_\theta(t)\vee F_{\theta,n}(t)}},
\end{align}
where $F_\theta(\cdot)$ is the cdf of $\mu_{\theta}$ and $F_{\theta,n}(\cdot)$ is the cdf of the projection $x\mapsto \langle x, \theta \rangle_{\mathcal{H}}$ under $\mu_n$. These bounds are of independent interest and are well-studied in the literature; an incomplete list is \cite{wellner1978limit, alexander1985rates, alexander1987central, pollard1981limit, BARTLETT199955, gine2003ratio, bartl2023variance} and the references therein. To the best of our knowledge, connections of these bounds and the statistical behaviour of the max-sliced Wasserstein distances $\overline{\mathcal{W}}_p(\mu_n,\mu)$ have been made fairly recently; we refer to \cite{bartl2022structure, olea2022generalization} for the max-sliced Wasserstein distance and \cite{manole2022minimax} for the sliced Wasserstein distance. In this note, we extend these to the case of general $p\ge 1$.

Once bounds on the rates $\E[\overline{\mathcal{W}}_p(\mu_n,\mu)^p]$ are established for a finite-dimensional $\mathcal{H}$, it is natural to ask how much of this structure extends to infinite dimensional spaces. In this note we focus our analysis on \emph{reproducing kernel Hilbert spaces} (RKHS) that have a representation as weighted $L^2$-spaces
\begin{align}\label{eq:rkhs_intro}
 \mathcal H=\left\{f\in L^2(\Omega, m): \sum\limits_{j=1}^{\infty}\frac{\langle f,\psi_j\rangle^2_{L^2}}{\lambda_j}<\infty\right\}
\end{align}
for some measure space $(\Omega, m)$, a non-increasing sequence $(\lambda_j)_{j\in \N}$ converging to zero for $j\to \infty$, and a sequence of functions $(\psi_j)_{j\in \N}$ in $L^2(\Omega, m)$ with scalar product $\langle \cdot, \cdot\rangle_{L_2}$. Going back at least to \cite{poggio1989theory, girosi1998equivalence, scholkopf1997comparing}, RKHS are by now indispensable tools in statistics and machine learning, in particular to kernel-based regression or the kernel support vector machine. We refer to \cite{williams1998prediction, van2004benchmarking, wahba1990spline, cai2006prediction, shawe2004kernel, hofmann2008kernel} for an overview of current developments in this area. Studying optimal transport in RKHS is not new; see e.g.\ \cite{gelbrich1990formula,cuesta1996lower,ZZWN, Oh_2020,nath2020statistical, wang2021two} and the references therein. To the best of our knowledge, however, there are no results for the max-sliced Wasserstein distance in infinite dimensional spaces available in the literature.

The representation \eqref{eq:rkhs_intro} is known as Mercer's theorem \cite{schmidt1908theorie, mercer1909xvi} and holds for a large class of RKHS. We refer to Section \ref{sec:rkhs} for a detailed discussion. We also remark that all our main results stated below can easily be transferred to a (non-weighted) $L^2$-space. We detail this alternative in Section \ref{sec:lit_review}. There we also give a detailed comparison to the work of \cite{lei2020convergence}, which gives rates for the (standard) Wasserstein distance $\mathcal{W}_p(\mu_n,\mu)$ for probability measures in $L^2$-spaces.

In this note, we consider two main scenarios, which are very classical in the study of RKHS, see e.g.\ \cite{yang2020function} or the comparison to \cite{lei2020convergence} in Section \ref{sec:lit_review}: if the eigenvalues $\{\lambda_j\}_{j\in \N}$ decay exponentially, i.e.\ $\lambda_j \le C\exp(-cj^\gamma)$ for some constants $\gamma, c,C>0$ and the weighted moment 
\begin{align}\label{eq:moment_intro}
\widetilde{M}_s(\mu):=\int \Big(\sum_{j=1}^{\infty}\frac{1}{\lambda_j^2}\langle x, \psi_j\rangle^2_{L^2} \Big)^{\frac{s}{2}}\,\mu(dx),
\end{align}
is finite for some $s>2p$, then
\begin{align*}
\E[\overline{\mathcal{W}}_p(\mu_n, \mu)^p]\lesssim  \sqrt{\frac{1}{n}}.
\end{align*}
As before, $\lesssim$ hides logarithmic factors of $n$ and a constant depending on $p,s,\gamma, \widetilde M_s(\mu)$ only. 
If $\{\lambda_j\}_{j\in \N}$ decay only polynomially, i.e.\ $\lambda_j\leq C j^{-\gamma}$ for all $j\geq 1$ for some constants $C>0$ and $\gamma>1$, then 
\begin{align*}
\E[\overline{\mathcal{W}}_p(\mu_n, \mu)^p]\lesssim  \sqrt{\frac{1}{n^{1-\frac{1}{p\gamma}}}}.
\end{align*}
We also detail the corresponding uniform ratio bounds for both of these scenarios under slightly stronger assumptions on $\mu$, see  Theorem \ref{thm:unif_exp} and \ref{thm:unif_poly} in Section \ref{sec:main} for the exact statements.

\subsection{Related Work}

\textbf{(Max-)sliced Wasserstein distances.} Sliced Wasserstein distances \cite{rabin2011wasserstein, bonneel2015sliced, han2023sliced} and max-sliced Wasserstein distances, meaning distances between probability distributions that rely on computing the average/maximum of the standard Wasserstein distance between one-dimensional projections, have been studied recently in the literature, e.g.\ \cite{kolouri2019generalized, kengo12022statistical, niles2022estimation, lin2021projection, bartl2022structure, deshpande2019max, paty2019subspace}. 
In particular, \cite{olea2022generalization} use the max-sliced Wasserstein distance to analyze the out-of-sample prediction error of the square-root LASSO and related estimators.

\textbf{Kernel Methods.} Kernel methods are by now popular in the statistics and machine learning literature as a means to, for example, analyze complicated data types in feature space \cite{lodhi2002text, Kondor2002DiffusionKO, ben2005kernel, cortes2004rational, cortes2003lattice} or to lift data into higher dimensions for easier analysis. In general, because the data is first mapped into the RKHS through an implicit feature map, we do not have access to the pushforward measures on the RKHS and, generally, analysis in RKHS cannot be performed using similar methods as can be used to study the original finite-dimensional spaces. 
We emphasize that our results are still meaningful for analysis of methods even if the feature map is implicit and the practitioner has access only to data-dependent kernel functions that characterize the similarity between observations. Indeed, consider kernel regression where one has access to i.i.d.\ observations $\{\tilde y_i, \tilde x_i\}_{i=1}^n$ and wishes to find an estimator $\widehat{\beta}$ solving $\arg\min_{\beta} \sum_{i=1}^n(\tilde y_i - K(\tilde x_i, \beta))^2$ using some kernel $K$. If we consider prediction at a new observation $\tilde x_{new}$ with the estimator $K( \widehat{\beta}, \tilde x_{new})$, our results can show concentration to the truth $K(\beta_0, \tilde x_{new})$ (under the assumption the data was generated from such a model using $\beta_0$). To see this, notice that $K( \widehat{\beta}, \tilde x_{new}) = \tilde y \widehat{K}^{\dagger} \widehat{K}(\tilde X^n, \tilde x_{new})$ where we define $\widehat{K}_{i,j} := K(\tilde x_i, \tilde x_j)$ and $\widehat{K}(\tilde X^n, \tilde x_{new})_{i} := K(\tilde x_i, \tilde x_{new})$ and we can study concentration of these objects to their population values by noticing that, for example, with $\tilde X$ representing the true distribution and $\tilde X^n$ the empirical distribution of the data $\{\tilde x_i\}_{i=1}^n$,
\[\lvert K(\tilde X^n, \tilde x_{new}) - K(\tilde X, \tilde x_{new}) \lvert \]
can be bounded using $\overline{\mathcal{W}}_2.$ We refer to Section \ref{sec:numerical} for further illustration of our results for RKHS with Gaussian kernel.

Regarding previous work studying optimal transport in RKHS, \cite{cuesta1996lower} and \cite{gelbrich1990formula} provide expressions for the (standard) Wasserstein distance between Gaussian measures on Hilbert spaces. \cite{ZZWN} constructs optimal transports in an RKHS under Gaussianity assumptions, \cite{han2023sliced} derives the notion of the sliced Wasserstein distance for measures in separable Hilbert spaces, and \cite{wang2021two} considers a two-sample test on finite dimensional RKHS.

\textbf{Concentration for Wasserstein Metrics.} Departing from the works \cite{dereich2013constructive, fournier2015rate}, which establish finite sample bounds on the Wasserstein metric, a recent focus in the literature has been to establish finite sample bounds on  variants of the Wasserstein metric: see, for example, \cite{boissard2014mean, singh2018minimax, weed2019sharp, niles2022estimation, lei2020convergence, chizat2020faster} and the references therein.
As mentioned previously, the Wasserstein metric suffers from the curse of dimensionality, while faster rates of convergence for the max-sliced Wasserstein metric in Euclidean space were first observed in \cite{niles2022estimation} for subgaussian probability measures and in \cite{lin2021projection} under a projective Poincar\'e/Bernstein inequality. More recently, in Euclidean space, \cite{bartl2022structure} have obtained sharp rates for $r=2$ and isotropic distributions. We also refer to  \cite{kengo12022statistical}, which derives sharp rates for the max-sliced Wasserstein metric in Euclidean space for log-concave measures, that explicitly state the dependence on the dimension of the data.

\subsection{Notation}
Throughout this note we take $p,q\in [1,\infty)$. We write $M_q(\nu)=\int_\mathcal{H} \|x\|^q_{\mathcal{H}}\,\nu(dx)$ for the $q$th moment of a measure $\nu\in \mathcal{P}(\mathcal{H}),$ where $\mathcal{P}(\mathcal{H})$ denotes the space of probability measures on $\mathcal{H}$.  For any $\theta\in \mathcal{H}$, we write $F_\theta(t)=\mu(\{x\in \mathcal{H}: \langle x, \theta\rangle \le t\})$ and correspondingly $F_{\theta,n}(t)=\mu_n(\{x\in \mathcal{H}: \langle x, \theta\rangle \le t\})$. We use the standard convention $\nu(A):=\nu(\{x\in \mathcal{H}: x\in A\})$ for any Borel set $A\subseteq \mathcal{H}$ and often omit $\mathcal{H}$ in the integral bounds if the context is clear, e.g.\ $M_q(\nu)=\int \|x\|^q_{\mathcal{H}}\,\nu(dx)$.

We denote generic constants by 
$c,C>0$. These might change from line to line, with the convention that $C$ increases from line to line and $c$ decreases from line to line.

\subsection{Organization of the paper}
In Section \ref{sec:rkhs} we recall basic facts about RKHS, before stating the main results of this paper in Section \ref{sec:main}. We then compare to the setting and results of \cite{lei2020convergence} in Section \ref{sec:lit_review}, and give numerical examples in Section \ref{sec:numerical} for finite and infinite dimensional RKHS. All proofs are collected in Section \ref{sec:proofs} and Section \ref{sec:topology} includes some notes of the topology of generated by $ \overline{\mathcal{W}}_p$.

\section{Reproducing Kernel Hilbert spaces}
\label{sec:rkhs}

Let us now specify the assumptions we make on the RKHS $\mathcal{H}.$ If $\mathcal{H}$ is  $d$-dimensional, we simply identify $(\mathcal{H}, \|\cdot\|_{\mathcal{H}})$ with the Euclidean space $(\R^d, |\cdot|)$ for notational convenience and do not make any further assumptions. In the infinite dimensional case, we consider the following classical setting, see e.g.\ \cite[Chapter 11]{Paulsen_Raghupathi_2016}.

Let us recall that a RKHS $\mathcal{H}$ is a subset of functions $f:\Omega \to \R$, where $\Omega$ is an arbitrary set, satisfying the following properties:
\begin{itemize}
    \item $\mathcal{H}$ is a Hilbert space equipped with a scalar product $\langle \cdot, \cdot \rangle_{\mathcal{H}}$ and a norm $\|\cdot\|_{\mathcal{H}}$ and the operations addition $(f+g)(x):= f(x)+g(x)$ and scalar multiplication $(\lambda f)(x):=\lambda f(x)$.
    \item For any $z\in \Omega$, the evaluation functional $E_z$ defined via $E_z(f)=f(z)$ for all $f\in \mathcal{H}$ is bounded, i.e.\ there exists $M_z<\infty$ such that  
    $E_z(f)\le M_z \|f\|_{\mathcal{H}}$ holds for all $f\in \mathcal{H}$. By the Riesz representation theorem, this implies existence of a kernel $K:\Omega \times \Omega\to \R$ satisfying 
    \begin{align*}
    f(z)=E_z(f)=\langle f, K(z,\cdot)\rangle_\mathcal{H}, \qquad\forall f\in \mathcal{H}.
    \end{align*}
\end{itemize}

Throughout this note we assume that the kernel $K$ is continuous, symmetric, and  positive semi-definite and there exists a measure space $(\Omega, m)$ such that  $K\in L^2(\Omega\times\Omega,m\otimes m)$. Here $m\otimes m$ denotes the product measure.
Let $T_K: L^2(\Omega,m)\to L^2(\Omega,m)$ be the integral operator induced by $K$, in other words
\[
T_K f(z):=\int_{\Omega}K(z,z')f(z')\,m(dz'),\qquad\forall f\in L^2(\Omega,m).
\]
Then $T_K$ is Hilbert-Schmidt by \cite[Theorem \MakeUppercase{\romannumeral 6}.23]{reed_simon} and $T_K$ is compact by  \cite[Theorem \MakeUppercase{\romannumeral 6}.22 (e)]{reed_simon}. Because $K$ is symmetric, $T_K$ is self-adjoint, meaning that, for $f,g\in L^2(\Omega, m)$, 
\[
\langle f, T_Kg\rangle_{L^2} =\langle T_Kf,g\rangle_{L^2},
\]
where we recall that $\langle \cdot , \cdot \rangle_{L^2}$ denotes the inner product on $L^2(\Omega,m)$.
The operator $T_K$ has countable eigenvalues $\{\lambda_j\}_{j\in\N}$ that are non-increasing and satisfy $\lim_{j\to \infty}\lambda_j= 0$  by the Hilbert-Schmidt Theorem (\cite[Theorem \MakeUppercase{\romannumeral 6}.16]{reed_simon}). The corresponding eigenfunctions $\{\psi_j\}_{j\in\N}$ form a complete orthonormal basis of $L^2(\Omega, m)$. Using \cite[Theorem 11.18]{Paulsen_Raghupathi_2016}, the RKHS $\H$ can be written as a subset of $L^2(\Omega, m)$, i.e.\
\begin{equation}\label{eq: RKHS}
  \mathcal H=\left\{f\in L^2(\Omega, m): \sum\limits_{j=1}^{\infty}\frac{\langle f,\psi_j\rangle^2_{L^2}}{\lambda_j}<\infty\right\};  
\end{equation}
see \eqref{eq:rkhs_intro}.
Furthermore, by \cite[Theorem 11.18]{Paulsen_Raghupathi_2016}
we have
\begin{align*}
T_K f = \sum_{j=1}^\infty \sqrt{\lambda_j} \langle f, \psi_j\rangle_{L^2} \psi_j,
\end{align*}
and
\begin{align*}
K(z,z') = \sum_{j=1}^\infty \sqrt{\lambda_j} \psi_j(z) \psi_j(z').
\end{align*}
Lastly, the inner product $\langle \cdot,\cdot \rangle_{\H}$ can be written as
\begin{equation}
\label{eq:inner_prod}
\langle f,g\rangle_{\mathcal H}=\sum\limits_{j=1}^{\infty}\frac{1}{\lambda_j} \langle f,\psi_j \rangle_{L^2} \langle g,\psi_j \rangle_{L^2},
\end{equation}
the scaled eigenfunctions $\{\sqrt{\lambda_j}\psi_j\}_{j\in\N}$ form an orthogonal basis of the RKHS $\H$ and a feature mapping $\phi(z)\in\H$ is given by $\phi(z)=\sum\limits_{j=1}^{\infty}\sqrt{\lambda_j}\psi_j(z) \psi_j$ for any $z\in\Omega$. 

\begin{remark} In the case where $\Omega$ is compact subset of $\R^d$, $m$ is the Lebesgue measure on $\R^d$ and $\sup_{z\in\Omega}K(z,z)<\infty$, the above decomposition of $T_K$ follows from Mercer's theorem \cite{Paulsen_Raghupathi_2016}.
\end{remark}

\section{Main results}
\label{sec:main}

We work under the following standing assumption: 

\begin{assumption}\label{ass:s_moment}
There exists $s>2p$ such that $M_{s}(\mu)<\infty$.
\end{assumption}

It is well-known from the one-dimensional setting, that this assumption is (nearly) necessary to obtain rates of order $1/\sqrt{n},$ see e.g. \cite[Theorem 1]{fournier2015rate} and \cite[Corollary 7.17]{bobkov2019one}.

\subsection{The finite dimensional case}

\begin{theorem}[Concentration of $\overline{\mathcal{W}}_p$]\label{thm:conc_fin}
Under Assumption \ref{ass:s_moment} we have
\begin{align*}
&\P\left( \overline{\mathcal{W}}_p(\mu_n, \mu)^p \ge c\log(2n+1)^{\frac{p}{s}} \left[\sqrt{\frac{d}{n}} +\left(1+\sqrt{M_{s}(\mu) \vee M_{s}(\mu_n)}\right)\epsilon \right]\right)\\
&\le e^{-\frac{n\epsilon^2}{2}}+ 8e^{\log(2n+1)[2(d+1)-\frac{n\epsilon^2}{64}]},
\end{align*}
for all $\epsilon>0$. The constant $c$ depends only on $p$ and $s$.
\end{theorem}

Note that the only assumption we impose here is Assumption \ref{ass:s_moment}. In particular, unlike the bounds currently existing in the literature, no subgaussian tails, Bernstein tail conditions or Poincar\'e inequalities need to be checked. Let us also remark that the estimate contains the empirical moments $M_s(\mu_n)$. To obtain a deterministic bound, one can apply a union bound to isolate the event $\{|M_s(\mu)-M_s(\mu_n)|\ge \tilde \epsilon\},$ where $\tilde\epsilon>0.$
This gives the following corollary:

\begin{corollary}\label{cor:conc_fin}
Under Assumption \ref{ass:s_moment} we have 
\begin{align*}
&\P\left( \overline{\mathcal{W}}_p(\mu_n, \mu)^p \ge c\log(2n+1)^{\frac{p}{ s}} \left[\sqrt{\frac{d}{n}} +\left(1+\sqrt{M_s(\mu)+\tilde\epsilon} \right) \epsilon \right]\right)\\
&\le  9e^{2(d+1)\log(2n+1)-\frac{n\epsilon^2}{64}}+ \frac{2M_s(\mu)}{\tilde \epsilon},
\end{align*}
for all $\epsilon, \tilde\epsilon >0$. The constant $c$ depends only on $p$ and $s$.
\end{corollary}

Indeed, as soon as Assumption \ref{ass:s_moment} is satisfied, an application of Markov's inequality gives $$\P\left(\left|M_s(\mu)-M_s(\mu_n)\right|\ge \tilde\epsilon\right)\le 2M_s(\mu)/\tilde \epsilon.$$ If higher moments of $\mu$ are available, this bound can be improved, as detailed e.g., in \cite{larsson2023concentration}. For the convenience of the reader we restate a version adapted to our notation below:

\begin{lemma}[{\cite[Proposition 3.1]{larsson2023concentration}}]
We have the following:
\begin{enumerate}
\item Suppose $\int e^{a\left|x\right|^{s\beta}}\,\mu(dx)<\infty$ for some $a>0$ and $\beta \geq 1$. Then for all $n \in \mathbb{N}$ and $\tilde{\epsilon}>0$,
$$
\mathbb{P}\left(\left|M_s(\mu_n)-M_s(\mu) \right|>\tilde\epsilon\right) \leq C\big(e^{-c n \tilde{\epsilon}^2} \mathbf{1}_{\{\tilde{\epsilon} \leq 1\}}+e^{-c n \tilde{\epsilon}^\beta} \mathbf{1}_{\{\tilde{\epsilon}>1\}}\big),
$$
for some positive constants $c$ and $C$ that depend only on a, $\beta, \int |x|^s\,\mu(dx)$ and $\int e^{a\left|x\right|^{s\beta}}\,\mu(dx)$.

\item  Suppose $\int e^{a\left|x\right|^{s\beta}}\,\mu(dx)<\infty$ for some $a>0$ and $\beta \in(0,1)$. Then for all $n \in \mathbb{N}$ and $\tilde{\epsilon}>0$,
$$
\mathbb{P}\left(\left|M_s(\mu_n)-M_s(\mu)\right|>\tilde{\epsilon}\right) \leq C\left(e^{-c n \tilde{\epsilon}^2}+e^{-c(n \tilde{\epsilon})^\beta}\right),
$$
for some positive constants $c$ and $C$ that depend only on a, $\beta, \int |x|^s\mu(dx)$ and $\int e^{a\left|x\right|^{s\beta}}\mu(dx)$.

\item Suppose $\int \left|x\right|^{st}\,\mu(dx) <\infty$ for $t>2$. Then for all $n \in \mathbb{N}$ and $\tilde{\epsilon}>0$,
$$
\mathbb{P}\left(\left|M_s(\mu_n)-M_s(\mu)\right|>\tilde{\epsilon}\right) \leq e^{-c n \tilde{\epsilon}^2}+C n(n \tilde{\epsilon})^{-t},
$$
for some positive constants $c$ and $C$ that depend only on $t$ and $\int \left|x\right|^{st}\mu(dx)$.

\item  Suppose $\int \left|x\right|^{st}\,\mu(dx)<\infty$ for $t \in[1,2]$. Then for all $n \in \mathbb{N}$ and $\tilde{\epsilon}>0$,
$$
\mathbb{P}\left(\left|M_s(\mu_n)-M_s(\mu)\right|>\tilde{\epsilon}\right) \leq C n(n \tilde{\epsilon})^{-t},
$$
for some positive constant $C$ that depends only on $t$ and $\int \left|x\right|^{st}\mu(dx)$.
\end{enumerate}
\end{lemma}

From Corollary \ref{cor:conc_fin} we can derive the following:

\begin{theorem}[Expectation of $\overline{\mathcal{W}}_p$]\label{thm:exp_fin}
Under Assumption \ref{ass:s_moment} we have 
\begin{align*}
\E\left[\overline{\mathcal{W}}_p(\mu_n, \mu)^p\right]\le  C  \log(2n+1)^{\frac{p}{ s}+ \frac{1}{2}}  \sqrt{\frac{d}{n}}.
\end{align*}
The constant $C$ depends only on $p,s$, and $M_s(\mu)$.
\end{theorem}

As mentioned in the Introduction, these rates are optimal up to logarithmic factors.

All of the results above are based on the following well-known uniform ratio bounds, which we state here for completeness.

\begin{theorem}[Uniform ratio bounds] \label{thm:unif_finite}
For any $\epsilon>0$ we have
\begin{align*}
&\P\left(\sup_{(\mathbf{\theta},t)\in \H \times {\mathbb R}}\frac{|F_{\mathbf{\theta}}(t)-F_{\mathbf{\theta},n}(t)|}{\sqrt{F_\theta(t)\vee F_{\theta,n}(t)}}\ge \epsilon \right)\le 8e^{(d+1)\log(2n+1)-4n\epsilon^2}.
\end{align*}
\end{theorem}

\subsection{Exponential decay} \label{sec:expo_dec}

We now turn the case $\text{dim}(\mathcal{H})=\infty$. For this we first need to strengthen Assumption \ref{ass:s_moment}.
\begin{assumption}\label{ass:mu}
For some $s>2p$,
\begin{align}
 \widetilde{M}_s(\mu):=\int \Big(\sum_{j=1}^{\infty}\frac{1}{\lambda_j^2}\langle x, \psi_j\rangle^2_{L^2} \Big)^{\frac{s}{2}}\,\mu(dx) <\infty.
\label{eq:moments}
\end{align}
\end{assumption}

We also make the following assumption about  exponential decay of the eigenvalues of $K$:

\begin{assumption}\label{ass:exponential decay}
We have $\lambda_j\leq C\exp(-c j^{\gamma})$ for all $j\geq 1$, where the constants $\gamma, c, C>0$.     
\end{assumption}

At least on an informal level, Assumptions \ref{ass:mu} and \ref{ass:exponential decay} restrict the amount of probability mass, that $\mu$ can place outside of a $d$-dimensional subspace of $\mathcal{H}$; see Lemma \ref{lem:exp} for a precise mathematical statement.
Under this assumption one can then carefully trade off the finite-dimensional case against the decay of eigenvalues to get the following results:

\begin{theorem}[Concentration of $\overline{\mathcal{W}}_p$]\label{thm:conc_exp}
Under Assumptions \ref{ass:mu} and \ref{ass:exponential decay}, we have  
\begin{align*}
&\P\left( \overline{\mathcal{W}}_p(\mu_n, \mu)^p \ge c\log(2n+1)^{\frac{p}{s}+\frac{1}{2\gamma}} \left[\sqrt{\frac{1}{n}} +\left(2+\sqrt{M_{s}(\mu) \vee M_{s}(\mu_n)}\right) \epsilon \right] \right)\\
&\le  Ce^{c\log(2n+1)^{1+\frac{1}{\gamma}}-\frac{n\epsilon^2}{64}}+ \frac{C}{(n\epsilon^2)^{\frac{s}{2p}}},
\end{align*}
for all $\epsilon>0$. The constants $c,C$ depend only on $p,s,\gamma, \widetilde{M}_p(\mu)$.
\end{theorem}

As in the finite dimensional case, one can further estimate $M_s(\mu_n)$ to obtain a deterministic bound of $\overline{\mathcal{W}}_p(\mu_n, \mu)^p$. Furthermore, one can derive rates for the expectation of $\overline{\mathcal{W}}_p(\mu_n,\mu)$ from Theorem \ref{thm:conc_exp}.

\begin{theorem}[Expectation of $\overline{\mathcal{W}}_p$]
\label{thm:exp_exp}
Under Assumptions \ref{ass:mu} and \ref{ass:exponential decay}, we have 
\begin{align*}
\E\left[\overline{\mathcal{W}}_p(\mu_n, \mu)^p\right]\le   C\frac{\log(2n+1)^{\frac{p}{s}+\frac{1}{2}+\frac{1}{\gamma}}}{\sqrt{n}}.
\end{align*}
The constant $C$ depends only on $p,s, \gamma, M_s(\mu)$.
\end{theorem}

In order to obtain uniform ratio bounds in the infinite dimensional case, we need to introduce a third assumption.

\begin{assumption}\label{ass:exp2}
For all $d\geq 1$ we have $$\mu\Big(\sum_{j=d+1}^\infty \frac{1}{\lambda_j}  \langle x,\psi_j\rangle_{L^2}^2 >0 \Big)\leq C\exp(-c d^{\gamma}).$$ 
\end{assumption}

Assumption \ref{ass:exp2} is a strengthening of Assumptions \ref{ass:mu} and \ref{ass:exponential decay}, as these only imply 
\begin{align}\label{eq:bullshit}
\mu\Big(\sum_{j=d+1}^\infty \frac{1}{\lambda_j}  \langle x,\psi_j\rangle_{L^2}^2 >\epsilon \Big)\le \frac{C}{\epsilon^{s/2}} \exp(-csd^\gamma).
\end{align}
Indeed, this follows by Markov's inequality and Lemma \ref{lem:exp}. In particular, letting $\epsilon\downarrow 0,$ the right side of \eqref{eq:bullshit} explodes. In this sense, Assumption \ref{ass:exp2} is stricter than Assumptions \ref{ass:mu} and \ref{ass:exponential decay}. Indeed, under Assumption \ref{ass:exp2}, the mass $\mu$ assigns to the spaces $\text{span}(\{\sqrt{\lambda_j} \psi_j: j=d+1,\dots \})$ has to decrease exponentially in $d.$

\begin{theorem}[Uniform ratio bounds]\label{thm:unif_exp}
Under Assumption \ref{ass:exp2}, we have
\begin{align*}
&\P\left(\sup_{(\mathbf{\theta},t)\in \H \times {\mathbb R}}\frac{|F_{\mathbf{\theta}}(t)-F_{\mathbf{\theta},n}(t)|}{\sqrt{F_\theta(t)\vee F_{\theta,n}(t)}}\ge \epsilon \right)\le Ce^{C\log(2n+1)^{1+\frac{1}{\gamma}}-c n\epsilon^2},
\end{align*}    
for all $\epsilon>0$, where $c>0$ is an absolute constant and $C>0$ depends on $\gamma$.
\end{theorem}

\subsection{Polynomial decay}

Lastly, we discuss the case where the eigenvalues $\{\lambda_j\}_{j\in \N}$ decay polynomially, as stated in the following assumption:
\begin{assumption}\label{ass:polyn_decay}
We have $\lambda_j\leq C j^{-\gamma}$ for all $j\geq 1$, where $C>0$ and $\gamma>1$.
\end{assumption}

Similarly to the above, we get the following two results:

\begin{theorem}[Concentration of $\overline{\mathcal{W}}_p$]\label{thm:conc_poly}
Under Assumptions \ref{ass:mu} and \ref{ass:polyn_decay}, we have 
\begin{align*}
\E\left[\overline{\mathcal{W}}_p(\mu_n, \mu)^p\right]\le   C \frac{\log(2n+1)^{\frac{p}{s}+\frac12} }{n^{\frac{1}{2}-\frac{1}{1+p\gamma}}} .
\end{align*}
The constant $C$ depends only on $p,s, \gamma, M_s(\mu)$.
\end{theorem}

\begin{theorem}[Expectation of $\overline{\mathcal{W}}_p$] \label{thm:exp_poly}
Under Assumptions \ref{ass:mu} and \ref{ass:polyn_decay}, we have 
\begin{align*}
\E[\overline{\mathcal{W}}_p(\mu_n, \mu)^p]\le   C \frac{\log(2n+1)^{\frac{p}{s}} }{n^{\frac{1}{2}-\frac{1}{2p\gamma}}} .
\end{align*}
The constant $C$ depends only on $p,s, \gamma, M_s(\mu)$.
\end{theorem}

As in Section \ref{sec:expo_dec}, we need to strengthen Assumptions \ref{ass:mu} and \ref{ass:polyn_decay} as follows in order to obtain uniform ratio bounds.

\begin{assumption}\label{ass:poly2}
For some $\gamma>1$ and all $d\geq 1$, we have $$\mu\Big(\sum_{j=d+1}^\infty \frac{1}{\lambda_j}  \langle x,\psi_j\rangle_{L^2}^2 >0 \Big)\leq Cd^{-\gamma}.$$ 
\end{assumption}

\begin{theorem}[Uniform ratio bounds] \label{thm:unif_poly}
Under Assumption \ref{ass:poly2}, we have
\begin{align*}
&\P\left(\sup_{(\mathbf{\theta},t)\in \H \times {\mathbb R}}\frac{|F_{\mathbf{\theta}}(t)-F_{\mathbf{\theta},n}(t)|}{\sqrt{F_\theta(t)\vee F_{\theta,n}(t)}}\ge \epsilon \right)\le e^{C \log(2n+1)n^{\frac{1}{\gamma}} -\frac{c n\epsilon^2}{64}},
\end{align*}    
for all $\epsilon>0$, where $c>0$ is an absolute constant and $C>0$ depends on $\gamma$.
\end{theorem}

\section{Rates for $\overline{\mathcal{W}}_{p}(\mu_n,\mu)$ in $L^2$-spaces} \label{sec:lit_review}

We now detail the implications of our results on the rates of $\overline{\mathcal{W}}_p(\mu_n, \mu)$ for probability measures $\mu$ on the space $L^2(\Omega, m)$, where $(\Omega,m)$ is a measure space. Abusing notation, we still denote an orthonormal basis of $L^2(\Omega, m)$ by $\{\psi_j\}_{j\in\N}$. We now reformulate our assumption on the integrability of $\mu$. In fact, Assumption \ref{ass:mu} now reads:

\begin{assumption}\label{ass:lei}
There exists a decreasing sequence $\{\tau_j\}_{j\in \N}$ of non-negative numbers such that, for some $s>2p$, we have
$$\int \Big( \sum_{j=1}^\infty  \frac{\langle x, \psi_j\rangle_{L^2}^2}{\tau_j^2} \Big)^{\frac{s}{2}} \,\mu(dx)<\infty. $$
\end{assumption}

Here the sequence $(\tau_j)_{j\in \N}$ plays the role of  $(\lambda_j)_{j\in \N}$ and allows us to make a direct connection to \cite{lei2020convergence}: in particular, \cite[Section 4.2]{lei2020convergence} covers the exponential decay $\tau_j=C\exp(-cj^\gamma)$, while \cite[Section 4.1]{lei2020convergence} discusses the polynomial decay rates $\tau_j=Cj^{-\gamma}$. 

We obtain the following meta-result:

\begin{mtheorem}
In Theorems \ref{thm:conc_exp},  \ref{thm:exp_exp}, \ref{thm:conc_poly} and \ref{thm:exp_poly} replace Assumption \ref{ass:s_moment}  by Assumption \ref{ass:lei} and $\{\lambda_j\}_{j\in \N}$ by $\{\tau_j\}_{j\in \N}$. Then the results remain unchanged, where $\overline{\mathcal{W}}_p(\mu_n,\mu)$ is the max-sliced Wasserstein distance between $\mu_n$ and $\mu$ on $L^2(\Omega, m).$
\end{mtheorem}

As an example, Theorem \ref{thm:exp_exp} in $L^2(\Omega,m)$ space reads like this:

\begin{theorem}[Concentration of $\overline{\mathcal{W}}_p$]
Assume $\tau_j\le C\exp(-cj^\gamma)$ for all $j\in \N$ and constants $c, C, \gamma>0$.
Under Assumption \ref{ass:lei} for $\{\tau_j\}_{j\in \N}$, we have 
\begin{align*}
&\P\left( \overline{\mathcal{W}}_p(\mu_n, \mu)^p \ge c\log(2n+1)^{\frac{p}{s}+\frac{1}{2\gamma}} \left[\sqrt{\frac{1}{n}} +\left(2+\sqrt{M_{s}(\mu) \vee M_{s}(\mu_n)}\right) \epsilon \right] \right)\\
&\le Ce^{C\log(2n+1)^{1+\frac{1}{\gamma}}-\frac{n\epsilon^2}{64}}+ \frac{C}{n^{\frac{s}{2p}} \epsilon^2}.
\end{align*}
The constants $c,C$ depend only on $p,s,\gamma, \widetilde{M}_p(\mu)$.
\end{theorem}

\section{Numerical examples}\label{sec:numerical}

In this section, we carry out a number of numerical tests of our main results about the rate of decay of $\E[\overline{\mathcal{W}}_p(\mu_n,\mu)]$, see Theorems \ref{thm:exp_fin}, \ref{thm:exp_exp} and \ref{thm:exp_poly}. We sometimes compute $\E[\overline{\mathcal{W}}_p(\mu_n,\nu_n)]$ instead of $\E[\overline{\mathcal{W}}_p(\mu_n,\mu)]$, where $\mu_n$ and $\nu_n$ are two empirical measures of independent samples from $\mu=\nu$. By the triangle inequality, we estimate 
\begin{align}\label{eq:rates_triangle}
\E[\overline{\mathcal{W}}_p(\mu_n,\nu_n)] &\le 
\E[\overline{\mathcal{W}}_p(\mu_n,\mu)] +\E[\overline{\mathcal{W}}_p(\nu_n,\nu)]=2 \E[\overline{\mathcal{W}}_p(\mu_n,\mu)].
\end{align}
In consequence, our results also yield upper bounds for $\E[\overline{\mathcal{W}}_p(\mu_n,\nu_n)]$. Computing rates for $\overline{\mathcal{W}}_p(\mu_n,\nu_n)$ instead of $\overline{\mathcal{W}}_p(\mu_n,\mu)$ is a standard technique to speed up computations, which is e.g.\ applied in \cite[Figure 1]{lin2021projection}.
The Python implementation of our experiments can be found at \href{https://github.com/Han-Ruiyu/rkhs-p-max-sw}{https://github.com/Han-Ruiyu/rkhs-p-max-sw}. Below, we first discuss convergence rates for measures $\mu$ on a finite dimensional space $\mathcal{H}$ and then consider $\mu$ on an RKHS with a Gaussian kernel.

\subsection{The finite dimensional case}

If $\text{dim}(\mathcal{H})<\infty$, then Theorem \ref{thm:exp_fin} yields the upper bound
\begin{align}\label{eq:cindy}
\E[\overline{\mathcal{W}}_2(\mu_n,\mu)]\le \E[\overline{\mathcal{W}}_2(\mu_n,\mu)^2]^{\frac{1}{2}} \lesssim \Big(\frac{d}{n}\Big)^{\frac14}.
\end{align}
We study this upper bound empirically for two distributions:

\begin{itemize}
    \item $\mu=\mathcal{N}(a,\Sigma)$ is a Gaussian distribution, where $a\in\mathbb R^d$ denotes the mean and $\Sigma\in\mathbb R^{d\times d}$ is the covariance matrix. It is well known that Gaussian distributions have moments of all orders; thus, $\mu$ satisfies Assumption \ref{ass:s_moment} for any $p\geq 1$. In the case $\mu=\mathcal{N}(0, I)$, we e.g.\ have
    \begin{align*}
    \int |x|^s\,\mu(dx)\le  d^{s} \Big( 2^{s/2}\frac{\Gamma((s+1)/2)}{\sqrt{\pi}} \Big),
    \end{align*}
    where $\Gamma(z)=\int_0^\infty t^{z-1}e^{-t}\,dt$ denotes the Gamma function.
    We plot a $n \mapsto \E[\overline{\mathcal{W}}_2(\mu_n,\mu)]$ for $\mu=\mathcal{N}(0,I)$ in Figure \ref{fig:1} and for $\mu=\mathcal{N}(1, \hat{\Sigma})$ in Figure \ref{fig:2}, where 
    $$
    1 = (1, \dots, 1)^T \quad \text{and} \quad \hat{\Sigma} = \begin{bmatrix}
    1 & 1/2 & 1/2 &\dots &1/2\\
    1/2 & 1 & 1/2 &\dots &1/2\\
    1/2 & 1/2 & 1 &\dots &1/2\\
    \vdots &\vdots &\vdots &\ddots &\vdots\\
    1/2 & 1/2 & 1/2 &\dots & 1
    \end{bmatrix}.
    $$
    In both cases, we average our estimates over $100$ Monte Carlo runs and plot the results for dimensions $d=2,4,8$ and compare to \eqref{eq:cindy}.

    \item $\mu=\text{Pareto}(a) \otimes \cdots\otimes \text{Pareto}(a),$ where $\otimes$ denotes the product measure and $\text{Pareto}(a)$ denotes the one-dimensional Pareto distribution with shape parameter $a\in (0,\infty)$ and cdf $F(x)=1-\frac{1}{x^a}$ for all $x \ge 0.$ By direct computation, we have 
    \begin{align*}
    \int |x|^s\,\mu(dx)\le d^s \Big(\frac{a}{a-s}\Big) \quad \text{ for }s\in (0,a) \quad \text{and} \quad 
    \int |x|^s\,\mu(dx)=\infty \quad \text{ for }s\geq a.
    \end{align*}
    This implies that $\mu$ does not satisfy the Bernstein tail condition
    \begin{align*}
    \int |x|^s\,\mu(dx) \leq \frac{1}{2}\sigma^2 s! V^{s-2} \quad \text{ for all }s\geq 2,
    \end{align*}
    where $V,\sigma >0$,
    but it satisfies Assumption \ref{ass:s_moment} for $p<a/2$.
    Recalling \eqref{eq:rates_triangle} we plot $n \mapsto \E[\overline{\mathcal{W}}_2(\mu_n,\nu_n)]$ for $\mu= \nu=\text{Pareto}(8) \otimes \cdots\otimes \text{Pareto}(8)$ in Figure \ref{fig:3}.
    We average our estimates over $200$ Monte Carlo runs and plot the results for dimensions $d=2,4,8$.
 \end{itemize}

\begin{figure}[htbp!] \label{fig:1}
    \centering
    \includegraphics[width=7cm]{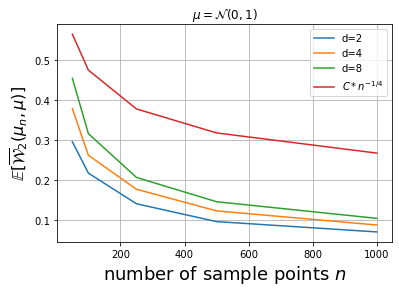}
    \includegraphics[width=7cm]{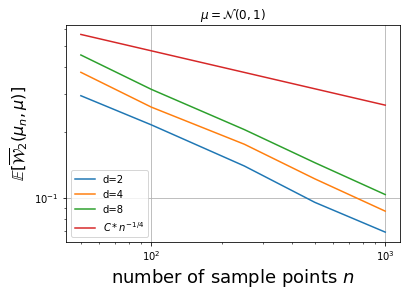}
    \caption{Plot and log-log plot of $n \mapsto \E[\overline{\mathcal{W}}_2(\mu_n,\mu)]$ where $\mu=\mathcal{N}(0,I_d)$, $d=2,4,8$. Results are averaged over $100$ Monte-Carlo runs.}
    \label{fig:1}
\end{figure}

\begin{figure}[htbp!]
    \centering
    \includegraphics[width=7cm]{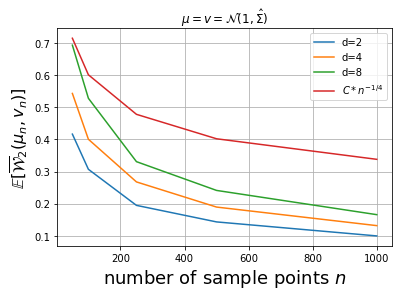}
    \includegraphics[width=7cm]{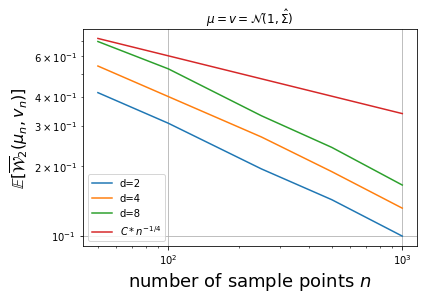}
    \caption{Plot and log-log plot of $n \mapsto \E[\overline{\mathcal{W}}_2(\mu_n,\nu_n)]$ where $\mu=\nu=\mathcal{N}(1,\hat{\Sigma})$. Results are averaged over $100$ Monte-Carlo runs.}
    \label{fig:2}
\end{figure}

\begin{figure}[htbp!]
    \centering
    \includegraphics[width=7cm]{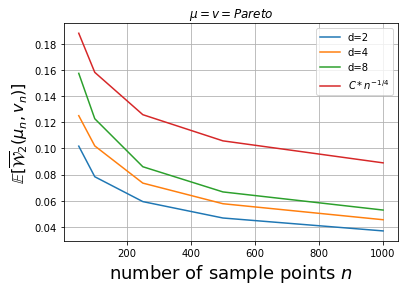}
    \includegraphics[width=7cm]{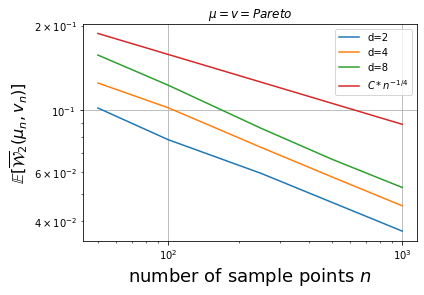}
    \caption{Plot and log-log plot of $n \mapsto \E[\overline{\mathcal{W}}_2(\mu_n,\nu_n)]$ where $\mu=\nu=\text{Pareto}(8) \otimes \cdots\otimes \text{Pareto}(8)$. Results are averaged over $200$ Monte-Carlo runs.}
    \label{fig:3}
\end{figure}

\subsection{The infinite dimensional case}
We now consider the case $\text{dim}(\mathcal{H})=\infty$. The following proposition can be found in \cite{williams2006gaussian}. For the convenience of the reader we provide a proof in Section \ref{sec:proof_5}.

\begin{proposition}\label{prop: gaussian_kernel_eigenfunctions}
Consider $\Omega=\R$ and $m= \mathcal{N}(0,\sigma^2)$ along with the Gaussian kernel $$K(z,z')=\exp(-(z-z')^2/2w^2),$$ for some $w>0.$ Then the eigenvalues $\{\lambda_j\}_{j\in \mathbb N}$ and eigenfunctions $\{\psi_j\}_{j\in \mathbb N}$ of $T_{K}:L^2(\mathbb R, m)\to L^2(\mathbb R, m)$ as defined in Section \ref{sec:rkhs} are given by
    \begin{equation}
\begin{split}
     \lambda_j &=  \sqrt{\frac{2a}{a+b+c}} \left( \frac{b}{a+b+c}\right)^j, \\
        \psi_j(z) &=\frac{1}{\sqrt{ \sqrt{\frac{a}{c}}  2^j j! }}\exp\left( -(c-a)z^2\right)H_j( \sqrt{2c}z), \label{eq: def_psi_j}
            \end{split}
    \end{equation}
    where
    \begin{equation}\label{eq: def_a_b_c}
        a=\frac{1}{4\sigma^2},\quad b=\frac{1}{2w^2},\quad c=\sqrt{a^2+2ab},
    \end{equation}
    and $H_j$ is the \emph{$j$-th order Hermite polynomial}, i.e.\ $$H_j(z)=(-1)^{j}e^{z^2}\frac{\partial^j}{\partial z^j}e^{-z^2},$$ for all $j\in \N$.
Moreover, $\{\psi_j\}_{j\in\N}$ form a complete orthonormal basis of $L^2(\mathbb R, m)$.
\end{proposition}

We now look at a specific example.

\begin{lemma}  \label{lem: rkhs_kernel_embed_eg}Let $m=\mathcal{N}(0,\sigma^2)$ and $K(z,z')=\exp(-(z-z')^2/(2w)^2)$. 
\begin{enumerate}
\item Define
\begin{align*}
\kappa:=  \frac{b}{a} =\frac{2\sigma^2}{w^2},
\end{align*}
and assume $\kappa\geq 4$.
Then the induced RKHS satisfies Assumption \ref{ass:exponential decay}.
\item Let $\tilde \mu=\mathcal{N}(0,\eta^2)$ and let $\phi:\mathbb R\to L^2(\Omega,m)$ be the feature map
\begin{align*}
\phi(z)=\sum\limits_{j=1}^{\infty}\sqrt{\lambda_j}\psi_j(z)\psi_j,
\end{align*} 
defined in Section \ref{sec:rkhs}. Set $\mu=\phi_{\#}\tilde \mu.$
Then $\mu$ satisfies Assumption \ref{ass:mu} for $s< 2\sigma^2/\eta^2$.
\end{enumerate}
\end{lemma}
\begin{figure}[htbp!]
    \centering
    \includegraphics[width=9cm]{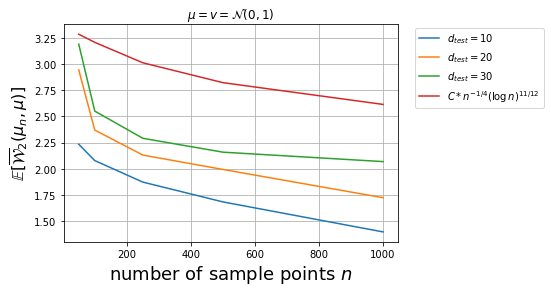}
    \includegraphics[width=9cm]{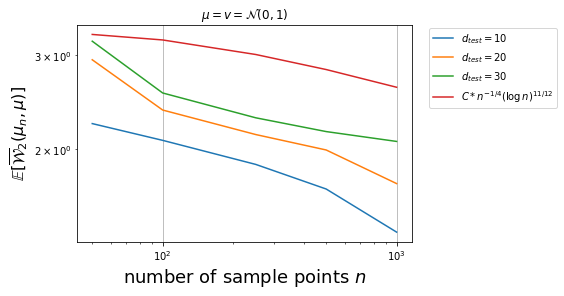}
    \caption{Plot and log-log plot of $n \mapsto \E[\overline{\mathcal{W}}_2(\mu_n,\nu_n)]$ where $\mu=\nu=\mathcal{N}(0,1)$. Results are averaged over $200$ Monte-Carlo runs.}
    \label{fig:rkhs_Gaussian}
\end{figure}

Take $m=\mathcal{N}(0,4)$ and $K(z,z')=\exp(-{(z-z')^2}/{4})$ and
$\mu=\mathcal{N}(0,1)$, so that $\sigma^2=4, w=1$ and $\eta^2=1$. Lemma \ref{lem: rkhs_kernel_embed_eg} is satisfied with $s=6$. According to Theorem \ref{thm:exp_exp} we then expect $\E \left[ \overline{\mathcal W}_2(\mu_n,\mu)\right]\lesssim n^{-1/4}\log(n)^{11/12}$.

In our numerical experiments we compute the truncated feature maps
\begin{equation}
\phi^{\text{test}}(z)=\sum\limits_{i=1}^{d_{\text{test}}}\sqrt{\lambda_j}\psi_j(z) \psi_j
\end{equation}
for $d_{\text{test}}\in\N.$
Figures \ref{fig:rkhs_Gaussian} shows our results for $d_{\text{test}}=10,20, 30.$ Similarly to the finite dimensional case we average the estimates of $\E [ \overline{\mathcal W}_2(\mu_n,\mu)]$ over 200 Monte Carlo runs.

\section{Proofs}\label{sec:proofs}

\subsection{An auxiliary result from empirical process theory}

Let us define a set $\mathcal{J}$ that contains the collection of halfspace functions in $\mathcal{H}$:
\begin{align*}
\mathcal{J}=\left\{\mathds{1}_{\left\{\langle x,\theta \rangle_{\H}\leq t \right\}} \, :\: ( \mathbf{\theta},t)\in\H\times\mathbb R \right\}.
\end{align*}   
Furthermore, for $n\in \N$, denoting $x_{1:n}:=(x_1, \dots, x_{n})\in (\R^d)^{n}$, we define 
\begin{align*}
\mathcal{S}_{\mathcal{J}}(x_{1:n}) := \text{card} \left(\left\{ (f(x_1), \dots, f(x_n)): f\in \mathcal{J}\right\}\right) \in \{1, 2, \ldots, 2^n\},
\end{align*}
and  the \emph{shatter coefficient} of $\mathcal{J}$ as
\begin{align*}
\mathcal{S}_{\mathcal{J}}(n) := \max_{x_{1:n}\in  (\R^d)^{n}}    \mathcal{S}_{\mathcal{J}}(x_{1:n}).
\end{align*}

Recall that $F_\theta(\cdot)$ is the cdf of $\mu_{\theta}$ (in other words, of $\langle x, \theta \rangle_{\mathcal{H}}$ under $\mu$) and $F_{\theta,n}(\cdot)$ is the cdf of $\langle x, \theta \rangle_{\mathcal{H}}$ under $\mu_n$.

The following lemma is a generalization of Theorem \ref{thm:unif_finite} and will be used in the proof of Theorems \ref{thm:unif_exp} and \ref{thm:unif_poly}.

\begin{lemma}\label{lem:ratio}
For any $\epsilon>0$, we have
\begin{align}\label{eq:ratio_general}
&\P\left(\sup_{(\mathbf{\theta},t)\in \H \times {\mathbb R}}\frac{|F_{\mathbf{\theta}}(t)-F_{\mathbf{\theta},n}(t)|}{\sqrt{F_\theta(t)\vee F_{\theta,n}(t)}}\ge \epsilon \right)  \le 8\E\left[\mathcal{S}_{\mathcal{J}}(X_{1:2n})\right] e^{-\frac{n\epsilon^2}{4}}.
\end{align}
If $\H$ is $d$-dimensional, then
\begin{align}\label{eq:vcdim}
\mathcal{S}_\mathcal{J}(n)\le (n+1)^{d+1}.
\end{align}
Therefore,
\begin{align}\label{eq:ratio}
&\P\left(\sup_{(\mathbf{\theta},t)\in \H \times {\mathbb R}}\frac{|F_{\mathbf{\theta}}(t)-F_{\mathbf{\theta},n}(t)|}{\sqrt{F_\theta(t)\vee F_{\theta,n}(t)}}\ge \epsilon \right)\le 8e^{(d+1)\log(2n+1)-\frac{n\epsilon^2}{4}}.
\end{align}
\end{lemma}

\begin{proof}
We first note the upper bound
\begin{equation}
\begin{split}
\label{lem6.1_eq1}
&\P\left(\sup_{(\mathbf{\theta},t)\in \H \times {\mathbb R}}\frac{|F_{\mathbf{\theta}}(t)-F_{\mathbf{\theta},n}(t)|}{\sqrt{F_\theta(t)\vee F_{\theta,n}(t)}}\ge \epsilon \right) \\
& \qquad \qquad \le  \P\left(\sup_{(\mathbf{\theta},t)\in \H \times {\mathbb R}}\frac{F_{\mathbf{\theta}}(t)-F_{\mathbf{\theta},n}(t)}{\sqrt{F_\theta(t)}}\ge \epsilon \right)  + \P\left(\sup_{(\mathbf{\theta},t)\in \H \times {\mathbb R}}\frac{F_{\mathbf{\theta},n}(t)-F_{\mathbf{\theta}}(t)}{\sqrt{F_{\theta,n}(t)}}\ge \epsilon \right).
\end{split}
\end{equation}
Considering the first term on the right hand side of \eqref{lem6.1_eq1}, the result in \cite[Ex.\ 3.3]{devroye2001combinatorial} states that for any measure $\mu$ with associated empirical measure $\mu_n$,
\begin{align}
\label{lem6.1_eq2}
\mathbb{P}\left(\sup _{f \in \mathcal{J}} \frac{\mu(f)-\mu_n(f)}{\sqrt{\mu(f)}}>\epsilon\right) \leq 4 \E\left[\mathcal{S}_{\mathcal{J}}(X_{1:2n})\right] e^{-\frac{n\epsilon^2}{4}},
\end{align}
and considering the second term on the right hand side of \eqref{lem6.1_eq1}, the result in \cite[Ex.\ 3.4]{devroye2001combinatorial} gives
\begin{align}
\label{lem6.1_eq3}
\mathbb{P}\left(\sup _{f \in \mathcal{J}} \frac{\mu_n(f)-\mu(f)}{\sqrt{\mu_n(f)}}>\epsilon\right) \leq 4 \E\left[\mathcal{S}_{\mathcal{J}}(X_{1:2n})\right] e^{-\frac{n\epsilon^2}{4}}.
\end{align}
Bounding \eqref{lem6.1_eq1} with \eqref{lem6.1_eq2} and \eqref{lem6.1_eq3}, we find the result \eqref{eq:ratio_general}, where we have used that for a given $(\mathbf{\theta},t)\in \H \times {\mathbb R}$, $F_\theta(t) = \mu(\langle \theta, x\rangle_{\mathcal{H}} \leq t) = \mu(f)$ for some $f \in \mathcal{J}$.

By \cite[Example 4.21]{wainwright2019high}, if $\mathcal{H}$ has dimension $d$, the VC-dimension of $\mathcal{J}$ is $(d+1)$. Thus, $\mathcal{S}_\mathcal{J}(n)\le (n+1)^{d+1}$, see e.g.\ \cite[Cor.\ 4.1]{devroye2001combinatorial}. Using this in \eqref{lem6.1_eq2}, we obtain
\begin{align*}
\mathbb{P}\left(\sup _{f \in \mathcal{J}} \frac{\mu(f)-\mu_n(f)}{\sqrt{\mu(f)}}>\epsilon\right) \leq 4 \E\left[\mathcal{S}_{\mathcal{J}}(X_{1:2n})\right] e^{-\frac{n\epsilon^2}{4}} \leq 4 \mathcal{S}_\mathcal{J}(2n)e^{-n\epsilon^2/4}
\le 4e^{(d+1) \log(2n+1)- \frac{n\epsilon^2}{4}}.
\end{align*}
A similar bound is attained for \eqref{lem6.1_eq3}, and plugging both into \eqref{lem6.1_eq1}, we have shown \eqref{eq:ratio}.
This concludes the proof.
\end{proof}

For ease of exposition we also record a corresponding result for the set 
\begin{align*}
\mathcal{I}= \left\{\mathds{1}_{\{\langle x,\theta \rangle_{\H}\leq t \}} \, :\: ( \mathbf{\theta},t)\in\H\times\mathbb R \right \}\cup \left\{\mathds{1}_{\{\langle x,\theta \rangle_{\H}> t\}} \, :\: ( \mathbf{\theta},t)\in\H\times\mathbb R \right\}.
\end{align*}

\begin{lemma}\label{lem:add}
If $\H$ is $d$-dimensional, then
\begin{align*}
\mathbb{P}\left(\sup _{f \in \mathcal{I}} \frac{|\mu(f)-\mu_n(f)|}{\sqrt{\mu(f)\vee \mu_n(f)}}>\epsilon\right) \leq 8e^{2(d+1)\log(2n+1)-\frac{n\epsilon^2}{4}}.
\end{align*}
\end{lemma}
\begin{proof}
The proof follows exactly as that for Lemma~\ref{eq:ratio_general}, noting that $\mathcal{S}_\mathcal{I}(n)\le (n+1)^{2(d+1)}$ by \cite[Example 4.21]{wainwright2019high}.
\end{proof}

\subsection{Proofs for the finite dimensional case}

As a warm up, we start with the setting where $\mathcal{H}$ is $d$-dimensional for some $d\ge1$. We recall from Section \ref{sec:main}, that we identify $(\mathcal{H},\|\cdot\|)$ with the Euclidean space $(\R^d, |\cdot|).$ In particular, the results in this section hold without the assumptions made in Section \ref{sec:rkhs}. They build on \cite[Proof of Theorem 3]{olea2022generalization}. For the convenience of the reader, we restate the results from Section \ref{sec:main} below, before proving them.

\begin{reptheorem}{thm:conc_fin}
Under Assumption \ref{ass:s_moment} we have
\begin{align*}
&\P\left( \overline{\mathcal{W}}_p(\mu_n, \mu)^p \ge c\log(2n+1)^{\frac{p}{s}} \left[\sqrt{\frac{d}{n}} +\left(1+\sqrt{M_{s}(\mu) \vee M_{s}(\mu_n)}\right)\epsilon \right]\right)\le  e^{-\frac{n\epsilon^2}{2}}+ 8e^{\log(2n+1)[2(d+1)-\frac{n\epsilon^2}{64}]},
\end{align*}
for all $\epsilon>0$. The constant $c$ depends only on $p$ and $s$.
\end{reptheorem}

\begin{proof}
By \cite[Proof of Theorem 3]{olea2022generalization}, we have
\begin{align}
\overline{\mathcal{W}}_p(\mu_n, \mu)^p \leq 2^p p \log (2 n+1)^{\frac{p}{s}}\left(I_1+\frac{\sqrt{M_{s}(\mu) \vee M_{s}(\mu_n)}}{s/2 -p} \log (2 n+1)^{-\frac{1}{2}} I_2\right),
\label{eq:42result}
\end{align}
where 
\begin{align*}
I_1 &= \sup _{(\theta, t) \in \R^d \times \mathbb{R}}\left|F_{\theta, n}(t)-F_{\theta}(t)\right|,\quad \text{and} \quad
I_2 = 2\sup _{(\theta, t) \in \R^d \times \mathbb{R}} \frac{\left(F_{\theta} (t)-F_{\theta, n}(t)\right)^{+}}{\sqrt{F_{\theta}(t)\left(1-F_{\theta, n}(t)\right)}} \vee \frac{\left(F_{\theta, n}(t)-F_{\theta}(t)\right)^{+}}{\sqrt{F_{\theta, n}(t)\left(1-F_{\theta}(t)\right)}}.
\end{align*}
By \cite[Proof of Lemma 1]{olea2022generalization}, we have
\begin{align}
\label{eq:I1bound}
\P\left(I_1\ge 180\sqrt{\frac{d+1}{n}} + \epsilon\right)\le e^{-\frac{n\epsilon^2}{2}}.
\end{align}
We now argue that
\begin{align}
\label{eq:details}
\frac{\left(F_{\theta} (t)-F_{\theta, n}(t)\right)^{+}}{\sqrt{F_{\theta}(t)\left(1-F_{\theta, n}(t)\right)}} \le  2 \sup _{f \in \mathcal{I}} \frac{|\mu(f)-\mu_n(f)|}{\sqrt{\mu(f)\vee \mu_n(f)}}.
\end{align}
For this we first consider the case $F_{\theta,n}(t)<1/2$. Then $1-F_{\theta,n}(t)>1/2$; therefore,
\begin{align*}
\frac{\left(F_{\theta} (t)-F_{\theta, n}(t)\right)^{+}}{\sqrt{F_{\theta}(t)\left(1-F_{\theta, n}(t)\right)}} \le  2 \frac{\left(F_{\theta} (t)-F_{\theta, n}(t)\right)^{+}}{\sqrt{F_{\theta}(t)}} \le 2 \sup _{f \in \mathcal{I}} \frac{|\mu(f)-\mu_n(f)|}{\sqrt{\mu(f)\vee \mu_n(f)}}.
\end{align*}
On the other hand, if $F_{\theta,n}(t)\ge 1/2$ we can without loss of generality assume that $F_\theta(t)\ge F_{\theta,n}(t)$ (otherwise the expression is zero). Then we have $F_{\theta}(t)\ge 1/2$, and so 
\begin{align*}
\frac{\left(F_{\theta} (t)-F_{\theta, n}(t)\right)^{+}}{\sqrt{F_{\theta}(t)\left(1-F_{\theta, n}(t)\right)}} \le  2 \frac{\left(F_{\theta} (t)-F_{\theta, n}(t)\right)^{+}}{\sqrt{1-F_{\theta,n}(t)}}.
\end{align*}
Now we note that the above expression is equal to 
\begin{align*}
2 \frac{\left([1-F_{\theta,n} (t)]-(1-F_{\theta}(t)]\right)^{+}}{\sqrt{1-F_{\theta,n}(t)}}.
\end{align*}
Lastly, as $1-F_{\theta,n}(t)=\mu_{n}(\langle x, \theta \rangle > t)$ and $1-F_\theta(t)= \mu (\langle x, \theta\rangle >t)$, we again see that 
\begin{align*}
2 \frac{\left([1-F_{\theta,n} (t)]-(1-F_{\theta}(t)]\right)^{+}}{\sqrt{1-F_{\theta,n}(t)}}\le 2 \sup _{f \in \mathcal{I}} \frac{|\mu(f)-\mu_n(f)|}{\sqrt{\mu(f)\vee \mu_n(f)}},
\end{align*}
by definition of $\mathcal{I}.$ This shows \eqref{eq:details}. Using a symmetric argument for 
$$\frac{\left(F_{\theta, n}(t)-F_{\theta}(t)\right)^{+}}{\sqrt{F_{\theta, n}(t)\left(1-F_{\theta}(t)\right)}}\le  2 \sup _{f \in \mathcal{I}} \frac{|\mu(f)-\mu_n(f)|}{\sqrt{\mu(f)\vee \mu_n(f)}},$$
we find
\begin{align*}
I_2 \le 4 \sup _{f \in \mathcal{I}} \frac{|\mu(f)-\mu_n(f)|}{\sqrt{\mu(f)\vee \mu_n(f)}}.
\end{align*}
By Lemma \ref{lem:add}, we thus conclude
\begin{align}
\label{eq:I2bound}
\P( I_2 \ge \epsilon) \le 8e^{2(d+1)\log(2n+1)-\frac{n\epsilon^2}{64}}.
\end{align}

In conclusion, using \eqref{eq:I1bound} and \eqref{eq:I2bound}, we find
\begin{align*}
&\P\left( \overline{\mathcal{W}}_p(\mu_n, \mu)^p \ge 2^pp\log(2n+1)^{\frac{p}{s}} \left[180\sqrt{\frac{d+1}{n}} +\epsilon+ \frac{\sqrt{M_{s}(\mu) \vee M_{s}(\mu_n)}}{s / 2-p} \epsilon \right]\right)\\
&\stackrel{\eqref{eq:42result}}{\le} \P\left( I_1+\frac{\sqrt{M_{s}(\mu) \vee M_{s}(\mu_n)}}{s/2 -p} \log (2 n+1)^{-\frac{1}{2}} I_2 \ge 180\sqrt{\frac{d+1}{n}} +\epsilon+ \frac{\sqrt{M_{s}(\mu) \vee M_{s}(\mu_n)}}{s / 2-p} \epsilon \right)\\
&\le \P\left( I_1 \ge 180\sqrt{\frac{d+1}{n}} +\epsilon \right) + \P\left(  I_2 \ge \log (2 n+1)^{\frac{1}{2}} \epsilon \right)\\
&\le  e^{-\frac{n\epsilon^2}{2}}+ 8e^{\log(2n+1)[2(d+1)-\frac{n\epsilon^2}{64}]}.
\end{align*}
\end{proof}

Recall that $M_q(\nu)=\int_\mathcal{H} |x|^q \,\nu(dx)$ is the $q$th moment of a measure $\nu\in \mathcal{P}(\mathcal{H})$. Then, by Markov's inequality, we have 
\begin{align}
\P\left(\left|M_{ s}(\mu_n)-M_{ s}(\mu) \right|\ge \tilde{\epsilon}\right)\le \frac{2M_s(\mu)}{\tilde\epsilon}.
\label{eq:moment_bound}
\end{align}
This yields the following corollary.

\begin{repcorollary}{cor:conc_fin}
Under Assumption \ref{ass:s_moment} we have 
\begin{align*}
&\P\left( \overline{\mathcal{W}}_p(\mu_n, \mu)^p \ge c\log(2n+1)^{\frac{p}{ s}} \left[\sqrt{\frac{d}{n}} +\left(1+\sqrt{M_s(\mu)+\tilde\epsilon} \right) \epsilon \right]\right)\\
&\le  9e^{2(d+1)\log(2n+1)-\frac{n\epsilon^2}{64}}+ \frac{2M_s(\mu)}{\tilde \epsilon},
\end{align*}
for all $\epsilon, \tilde\epsilon >0$. The constant $c$ depends only on $p$ and $s$.
\end{repcorollary}

\begin{proof}
This follows from \eqref{eq:moment_bound} and Theorem \ref{thm:conc_fin}, noting that  $$e^{-\frac{n\epsilon^2}{2}}+ 8e^{\log(2n+1)[2(d+1)-\frac{n\epsilon^2}{64}]} \leq  9e^{2(d+1)\log(2n+1)-\frac{n\epsilon^2}{64}}.$$
\end{proof}

We are now in a position to prove Theorem \ref{thm:exp_fin}. 

\begin{reptheorem}{thm:exp_fin}
Under Assumption \ref{ass:s_moment}, we have 
\begin{align*}
\E\left[\overline{\mathcal{W}}_p(\mu_n, \mu)^p\right]\le  C  \log(2n+1)^{\frac{p}{ s}+ \frac{1}{2}}  \sqrt{\frac{d}{n}}.
\end{align*}
The constant $C$ depends only on $p,s$, and $M_s(\mu)$.
\end{reptheorem}

\begin{proof}
Choosing $\epsilon=8v\sqrt{2\log(2n+1)(d+1)/n}$ and $\tilde \epsilon=v^4$ for $v>0$  in Corollary \ref{cor:conc_fin} (and adapting $c$ by a factor of $16\ge 8\sqrt{2(d+1)/d}$)  yields
\begin{align*}
&\P\bigg( \overline{\mathcal{W}}_p(\mu_n, \mu)^p \ge c\log(2n+1)^{\frac{p}{s}+\frac{1}{2}}\sqrt{\frac{d}{n}} \Big[1 +\big(1+\sqrt{M_{ s}(\mu)+v^4}\big) v \Big]\bigg)\\
&\le  9e^{2\log(2n+1)(d+1)(1-v^2)}+ \frac{2M_s(\mu)}{v^4}.
\end{align*}
Clearly, 
for $v\ge 1$, we have $\big(1+\sqrt{M_{ s}(\mu)+1}\big) v^{3} \geq \big(1+\sqrt{M_{ s}(\mu)+v^4}\big) v$ so that
\begin{align*}
&\P\bigg( \overline{\mathcal{W}}_p(\mu_n, \mu)^p \ge c\log(2n+1)^{\frac{p}{s}+\frac{1}{2}}\sqrt{\frac{d}{n}}\Big[ 1+\big(1+\sqrt{M_{ s}(\mu)+1}\big) v^{3}\Big]\bigg)\\
&\le \P\bigg( \overline{\mathcal{W}}_p(\mu_n, \mu)^p \ge c\log(2n+1)^{\frac{p}{s}+\frac{1}{2}}\sqrt{\frac{d}{n}}\Big[1 +\big(1+\sqrt{M_{ s}(\mu)+v^4}\big) v \Big]\bigg),
\end{align*}
and, subsuming the factor $\big(1+\sqrt{M_{ s}(\mu)+1}\big)$ in $c$, 
\begin{align}\label{eq:est}
\begin{split}
&\P\bigg( \overline{\mathcal{W}}_p(\mu_n, \mu)^p -c\log(2n+1)^{\frac{p}{s}+\frac{1}{2}}\sqrt{\frac{d}{n}} \ge  c \log(2n+1)^{\frac{p}{s}+\frac{1}{2}}\sqrt{\frac{d}{n}} v^{3}\bigg)\\
&\le  9e^{2\log(2n+1)(d+1)(1-v^2)}+ \frac{2M_s(\mu)}{v^4}.
\end{split}
\end{align}
Recall that for any integrable non-negative random variable $Z$, we have
\begin{align}\label{eq:int}
\E[Z]=\int_0^\infty \P\big( Z\ge u\big)\,du \leq 1+\int_1^{\infty}\P (Z\ge u)\,du.
\end{align}
Let $\tilde{Y}=Y\vee 0$, where
$$Y=\left( c \log(2n+1)^{\frac{p}{s}+\frac{1}{2}}\sqrt{\frac{d}{n}}\right)^{-1} \left[\overline{\mathcal{W}}_p(\mu_n, \mu)^p -c\log(2n+1)^{\frac{p}{s}+\frac{1}{2}} \sqrt{\frac{d}{n}}\right],$$ 
and notice that for $u>0$, we have $\P(\tilde{Y}\geq u)=\P(Y\geq u)$. Therefore,  we obtain from \eqref{eq:int}, 
\begin{align}
\label{eq:int2}
    \E[Y]\leq \E [\tilde{Y}]=\int_0^{\infty}\P(\tilde{Y}\geq u)\leq 1+ \int_1^{\infty}\P(\tilde Y\geq u) d u=1+ \int_1^{\infty}\P(Y\geq u)d u.
\end{align}
Then, we find
\begin{align*}
\E[\overline{\mathcal{W}}_p(\mu_n, \mu)^p]
&=   c \log(2n+1)^{\frac{p}{s}+\frac{1}{2}}\sqrt{\frac{d}{n}}\E[Y] + c\log(2n+1)^{\frac{p}{s}+\frac{1}{2}} \sqrt{\frac{d}{n}}\\
&\stackrel{\eqref{eq:int2}}{\le}   c \log(2n+1)^{\frac{p}{s}+\frac{1}{2}}\sqrt{\frac{d}{n}}\Big(1+\int_1^{\infty}\P (Y\ge u)\,du\Big) + c\log(2n+1)^{\frac{p}{s}+\frac{1}{2}} \sqrt{\frac{d}{n}}\\
&\stackrel{ \eqref{eq:est}}\le c\log(2n+1)^{\frac{p}{s}+\frac{1}{2}} \sqrt{\frac{d}{n}}\Big(2+ \int_1^\infty \Big[9e^{2\log(2n+1)(d+1)(1-u^{2/3})}+ \frac{2M_s(\mu)}{u^{4/3}}\Big]  \,du\Big)\\
&\le c\log(2n+1)^{\frac{p}{s}+\frac{1}{2}} \sqrt{\frac{d}{n}}\Big(2+ \int_1^\infty \Big[9e^{(1-u^{2/3})}+ \frac{2M_s(\mu)}{u^{4/3}}\Big]  \,du\Big).
\end{align*}
The claim follows as the last integral is finite and only depends on $M_s(\mu)$.
\end{proof}

\subsection{Proofs for the infinite dimensional case}

For any $d\in \N$, let  $\H^d$ to denote the $d$-dimensional subspace of $\H$ spanned by its first $d$ eigenfunctions, i.e.\ $$\H^d:=\text{span}\Big(\Big\{\sqrt{\lambda_j}\psi_j: j=1, \dots, d \Big\} \Big).$$ Then $\H^d$ is a Hilbert space with respect to the scalar product induced by $\H$. We denote by $\Pd: \H \to \H^d$ the canonical projection to $\H^d$, which is clearly Lipschitz continuous. For any $\nu\in \mathcal{P}(\H)$, the pushforward measure of $\nu$ through a function $f:\mathcal{H}\to \mathcal{X}$ for some Hilbert space $\mathcal{X}$ is denoted by $f_{\#}\nu$. More specifically, we define the pushforward measure $\nu^d:={\Pd}_{\#}\mu$, which is a probability measure on $\H^d$.  Furthermore, for any $\theta\in \mathcal{H}$, we denote by $\text{Proj}_\theta:\mathcal{H}\to \R$ the projection $x\mapsto \langle x, \theta\rangle_\mathcal{H},$ and by $\text{Proj}_{\theta \#}\mu$ the corresponding pushforward measure.

Our strategy to prove concentration will be to trade off results for finite-dimensional measures and the tail decay of $\{\lambda_j\}_{j=1}^\infty$. For this, we first note that by the triangle inequality
\begin{align}\label{eq:triangle}
\overline{\mathcal{W}}_p(\mu_n,\mu)^p\le 3^{p-1}\left(\overline{\mathcal{W}}_p( \mu_n,\mu_n^d)^p+ \overline{\mathcal{W}}_p(\mu_n^d,\mu^d)^p+\overline{\mathcal{W}}_p(\mu^d, \mu)^p\right),
\end{align}
where the above is true for any $d > 0$. For the middle term in \eqref{eq:triangle}, we recall that by Theorem \ref{thm:conc_fin}
\begin{align}\label{eq:finite_d}
\begin{split}
&\P\left( \overline{\mathcal{W}}_p(\mu_n^d, \mu^d)^p \ge c\log(2n+1)^{\frac{p}{s}} \left[\sqrt{\frac{d}{n}} +\left(1+\sqrt{M_{s}(\mu^d) \vee M_{s}(\mu_n^d)}\right) \epsilon \right]\right)\le  9e^{2(d+1)\log(2n+1)-\frac{n\epsilon^2}{64}}.
\end{split}
\end{align}
We start with the following observation which we will use to bound the first and third therm in \eqref{eq:triangle}:
\begin{lemma}\label{lem:fintoinf}
For any $\nu \in \mathcal{P}_p(\H)$ we have 
\begin{align*}
\overline{\mathcal{W}}_p(\nu^d,\nu )^p \le \int \|\Pd(x)-x\|_{\H}^p \, \nu(dx).
\end{align*}
\end{lemma}

\begin{proof}
We first note that for any $\pi\in \Pi(\nu^d,\nu)$ and any $\theta$ such that $\|\theta\|_{\H}=1$, we have $\Pts \pi \in \Pi(\Pts\nu ^d,{\Pts}\nu)$. Thus, using the Cauchy-Schwarz inequality,
\begin{align*}
\mathcal{W}_p(\Pts\nu^d,\Pts\nu)^p
&\le  \inf_{\pi \in \Pi(\nu^d, \nu)} \int |\langle x, \theta\rangle_{\mathcal{H}} - \langle y, \theta \rangle_{\mathcal{H}}|^p \, \pi(dx,dy) \\
&\le \int |\langle \Pd(x), \theta\rangle_{\mathcal{H}} - \langle x, \theta \rangle_{\mathcal{H}} |^p \, \nu(dx) \le \int \|\Pd(x)-x\|_{\H}^p \, \nu(dx).
\end{align*}
The claim follows as $\overline{\mathcal{W}}_p(\mu,\nu)=\sup_{\|\theta\|_{\H}=1}\mathcal{W}_p (\mu_\theta, \nu_\theta)$, see \eqref{eq:MSW}.
\end{proof}

\begin{lemma}\label{lem:fintoinf_prob}
We have
\begin{align*}
\P\Big(\int \|\Pd(x)-x\|_{\H}^p \, \mu_n(dx)\ge \epsilon \Big) \le \frac{1}{\epsilon^{\frac{s}{p}}} \int \|\Pd(x)-x\|_{\H}^s \, \mu(dx).   
\end{align*}
\end{lemma}

\begin{proof}
This follows from an application of Markov's inequality and Jensen's inequality, noting that $s>1$.
\end{proof}

\subsection{Proofs for exponential decay}

Recall that Assumption \ref{ass:exponential decay} states the following:  $$\lambda_j\leq C\exp\left(-c j^{\gamma}\right), \quad \text{ for all }j\geq 1, \text{ where }\gamma, c,C>0.$$

\begin{lemma}\label{lem:exp}
Under Assumptions \ref{ass:mu} and \ref{ass:exponential decay}, we have $$\int_{\H} \|\Pd (x)-x\|^q_{\H}\,\mu(dx) \leq C \exp\left(-c qd^\gamma\right),$$
for all $d\geq 1$ and all $q\in [p,s]$. The constant $c$ depends only on $\widetilde{M}_s(\mu)$, while the constant $C$ only depends on $q$ and $\widetilde{M}_s(\mu)$.
\end{lemma}

\begin{proof}
Using the representation in \eqref{eq:inner_prod}, we observe that for $d\ge 1$
\begin{align*}
\|x-\Pd(x)\|^2_{\H} = \langle x-\Pd(x), x-\Pd(x)\rangle_{\H}
= \sum\limits_{j=1}^{\infty}\frac{1}{\lambda_j} \langle x-\Pd(x),\psi_j \rangle^2_{L^2}  &=  \sum_{j=d+1}^{\infty}\frac{1}{\lambda_j}\langle x,\psi_j\rangle^2_{L^2} .
\end{align*}
Hence,
\begin{align*}
\int \|x-\Pd(x)\|^q_{\H} \,\mu(dx) 
&=\int  \Big(\sum_{j=d+1}^{\infty}\frac{1}{\lambda_j}\langle x,\psi_j\rangle^2_{L^2} \Big)^{\frac{q}{2}} \, \mu(dx) \\
&\leq \sup_{j\geq d+1} \lambda_j^{\frac{q}{2}} \int  \Big( \sum_{j=d+1}^{\infty}\frac{1}{\lambda_j^2}\langle x,\psi_j\rangle^2_{L^2} \Big)^{\frac{q}{2}} \,\mu(dx) \\
&\le \lambda_{d}^{\frac{q}{2}}  \int  \Big( \sum_{j=d+1}^{\infty}\frac{1}{\lambda_j^2}\langle x,\psi_j\rangle^2_{L^2} \Big)^{\frac{q}{2}} \,\mu(dx).
\end{align*}
As the integral is bounded by $1+\widetilde M_s(\mu)$ using the definition of $\widetilde M_s(\mu)$  in \eqref{eq:moments} and because $\lambda_{d}^{\frac{q}{2}}\le C^{\frac{q}{2}} \exp\left(-cqd^\gamma/2\right)$ by assumption, the claim follows.
\end{proof}

We are now ready for the proof of Theorem \ref{thm:conc_exp}.

\begin{reptheorem}{thm:conc_exp}
Under Assumptions \ref{ass:mu} and \ref{ass:exponential decay}, we have  
\begin{align*}
&\P\left( \overline{\mathcal{W}}_p(\mu_n, \mu)^p \ge c\log(2n+1)^{\frac{p}{s}+\frac{1}{2\gamma}} \left[\sqrt{\frac{1}{n}} +\left(2+\sqrt{M_{s}(\mu) \vee M_{s}(\mu_n)}\right) \epsilon \right] \right)\\
&\le  Ce^{c\log(2n+1)^{1+\frac{1}{\gamma}}-\frac{n\epsilon^2}{64}}+ \frac{C}{(n\epsilon^2)^{\frac{s}{2p}}},
\end{align*}
for all $\epsilon>0$. The constants $c,C$ depend only on $p,s,\gamma, \widetilde{M}_p(\mu)$.
\end{reptheorem}

\begin{proof}
Combining Lemmas \ref{lem:fintoinf} and \ref{lem:exp} with $q=p$, we obtain
\begin{align*}
\overline{\mathcal{W}}_p(\mu^d,\mu)^p \leq C \exp(-c pd^\gamma).
\end{align*}
Similarly, using Lemmas \ref{lem:fintoinf}, \ref{lem:fintoinf_prob} and \ref{lem:exp} with $q=s$ yields
\begin{align*}
\P\left(\overline{\mathcal{W}}_p(\mu^d_n,\mu_n)^p\ge \epsilon \right) \le \frac{C}{\epsilon^{\frac{s}{p}}}\exp(-csd^\gamma).
\end{align*}
Together with \eqref{eq:triangle} and \eqref{eq:finite_d}, we thus conclude for any $d > 0$ that
\begin{align*}
&\P\left( \overline{\mathcal{W}}_p(\mu_n, \mu)^p \ge c\log(2n+1)^{\frac{p}{s}} \left[\sqrt{\frac{d}{n}} +\left(2+\sqrt{M_{s}(\mu) \vee M_{s}(\mu_n)}\right) \epsilon \right]+ c \exp(-C pd^\gamma)\right)\\
&\le  Ce^{2(d+1)\log(2n+1)-\frac{n\epsilon^2}{64}} + \frac{C}{\epsilon^{\frac{s}{p}}}\exp(-csd^\gamma).
\end{align*}
Now, choosing $d=(\log n/(2Cp))^{\frac{1}{\gamma}}$ yields
\begin{align*}
&\P\left( \overline{\mathcal{W}}_p(\mu_n, \mu)^p \ge c\log(2n+1)^{\frac{p}{s}+ \frac{1}{2\gamma}} \left[\sqrt{\frac{1}{n}} +\left(2+\sqrt{M_{s}(\mu) \vee M_{s}(\mu_n)}\right) \epsilon \right] \right)\\
&\le  Ce^{c\log(2n+1)^{1+\frac{1}{\gamma}}-\frac{n\epsilon^2}{64}}+ \frac{C}{(n\epsilon^2)^{\frac{s}{2p}}}.
\end{align*}
\end{proof}

\begin{reptheorem}{thm:exp_exp}
Under Assumptions \ref{ass:mu} and \ref{ass:exponential decay} we have 
\begin{align*}
\E\left[\overline{\mathcal{W}}_p(\mu_n, \mu)^p\right]\le   C\frac{\log(2n+1)^{\frac{p}{s}+\frac{1}{2}+\frac{1}{\gamma}}}{\sqrt{n}}.
\end{align*}
The constant $C$ depends only on $p,s, \gamma, M_s(\mu)$.
\end{reptheorem}

\begin{proof}
We follow the same strategy as in the proof of Theorem \ref{thm:exp_fin}. First, we obtain a result similar to Corollary \ref{cor:conc_fin} using \eqref{eq:moment_bound} and Theorem~\ref{thm:conc_exp}:
\begin{align*}
&\P\left( \overline{\mathcal{W}}_p(\mu_n, \mu)^p \ge c\log(2n+1)^{\frac{p}{s}+\frac{1}{2\gamma}} \left[\sqrt{\frac{1}{n}} +\left(2+\sqrt{M_s(\mu)+\tilde \epsilon}\right) \epsilon \right]\right)\\
&\le  Ce^{c\log(2n+1)^{1+\frac{1}{\gamma}}-\frac{n\epsilon^2}{64}}+ \frac{C}{(n\epsilon^2)^{\frac{s}{2p}}} +\frac{2M_s(\mu)}{\tilde\epsilon}.  
\end{align*}
Choosing $\epsilon=8v\sqrt{C\log(2n+1)^{1+\frac{1}{\gamma}}/n}$ and $\tilde\epsilon=v^{\frac{s}{p}}$ we obtain
\begin{equation}
\begin{split}
&\P\left( \overline{\mathcal{W}}_p(\mu_n, \mu)^p \ge c\frac{\log(2n+1)^{\frac{p}{s}+\frac{1}{2}+\frac{1}{\gamma}}}{\sqrt{n}} \left[1+v\left(2+\sqrt{M_s(\mu)+v^{\frac{s}{p}}}\right)  \right]\right)\\
&\le  Ce^{C\log(2n+1)^{1+\frac{1}{\gamma}}(1-v^2)}+ \frac{C}{v^{\frac{s}{p}}}+ \frac{2M_s(\mu)}{v^{\frac{s}{p}}}. 
\label{eq:expec_bound}
\end{split}
\end{equation}
Following the same steps as in the proof of Theorem \ref{thm:exp_fin}, it remains to evaluate this inequality at $u=v^{1+\frac{s}{2p}}$ in order to bound the $\E[\overline{\mathcal{W}}_p(\mu_n, \mu)^p]$. Noting that $$\Big(1+\frac{s}{2p}\Big)^{-1} \frac{s}{p}=\frac{2s}{2p+s}>1,$$
we conclude that $$\int_1^\infty \left[Ce^{c\log(2n+1)^{1+\frac{1}{\gamma}}(1-u^{\frac{4p}{2p+s}})}+ \frac{C}{u^{\frac{2s}{2p+s}}}+ \frac{2M_s(\mu)}{u^{\frac{2s}{2p+s}}} \right]\,du<\infty,$$
is a constant depending on $M_s(\mu)$ only.
The claim follows.
\end{proof}

We now prove Theorem \ref{thm:unif_exp}. For this, recall that Assumption \ref{ass:exp2} states for all $d\geq 1$,
\begin{align}\label{eq:ass_exp2}
\mu\Big(\sum_{j=d+1}^\infty \frac{1}{\lambda_j}  \langle x,\psi_j\rangle_{L^2}^2 >0 \Big)\leq C\exp(-c d^{\gamma}).
\end{align}

\begin{reptheorem}{thm:unif_exp}
Under Assumption \ref{ass:exp2}, we have
\begin{align*}
&\P\left(\sup_{(\mathbf{\theta},t)\in \H \times {\mathbb R}}\frac{|F_{\mathbf{\theta}}(t)-F_{\mathbf{\theta},n}(t)|}{\sqrt{F_\theta(t)\vee F_{\theta,n}(t)}}\ge \epsilon \right)\le Ce^{C\log(2n+1)^{1+\frac{1}{\gamma}}-c n\epsilon^2},
\end{align*}    
for all $\epsilon>0$, where $c>0$ is an absolute constant and $C>0$ depends on $\gamma$.
\end{reptheorem}

\begin{proof}
By Lemma \ref{lem:ratio}, we have 
\begin{align}
\P\left(\sup_{(\mathbf{\theta},t)\in \H \times {\mathbb R}}\frac{|F_{\mathbf{\theta}}(t)-F_{\mathbf{\theta},n}(t)|}{\sqrt{F_\theta(t)\vee F_{\theta,n}(t)}}\ge \epsilon \right) \le 8\E\left[\mathcal{S}_{\mathcal{J}}(X_{1:2n})\right] e^{-c n\epsilon^2}.
\label{eq:lem61bound}
\end{align}
By the tower property, we have for any $d\ge 1$,
\begin{align*}
\E\left[\mathcal{S}_{\mathcal{J}}(X_{1:2n})\right]
&= \sum_{k=0}^{2n} 
\E\left[\mathcal{S}_{\mathcal{J}}(X_{1:2n})\Big| |\{j\in \{1, \dots, 2n\}:\, \Pd(X_j)\neq X_j\}|=k \right]\\
&\qquad\qquad \cdot \P\left(|\{j\in \{1, \dots, 2n\}: \Pd(X_j)\neq X_j\}|=k \right).
\end{align*}
On the set $\{|\{j\in \{1, \dots, 2n\}:\, \Pd(X_j)\neq X_j\}|=k\}$, we can bound $\mathcal{S}_{\mathcal{J}}(X_{1:2n})$ as follows: we have 
\begin{align*}
  \mathcal{S}_{\mathcal{J}}(\{ X_j: \, \Pd(X_j)= X_j, j\in \{1, \dots, n\} \})&=  \mathcal{S}_{\mathcal{J}}(\{ \Pd(X_j): \, \Pd(X_j)= X_j, j\in \{1, \dots, n\} \}) \\
  &\stackrel{\eqref{eq:vcdim}}{\le} (2n-k+1)^{d+1}.
\end{align*}
and trivially
\begin{align*}
 \mathcal{S}_{\mathcal{J}}(\{ X_j: \, \Pd(X_j)\neq X_j, j\in \{1, \dots, n\} \}) \le 2^k.
\end{align*}
Combining these,
\begin{align*}
\mathcal{S}_{\mathcal{J}}(X_{1:2n}) \le 2^{k} (2n-k+1)^{d+1}.
\end{align*}
By \eqref{ass:exp2} with the bound ${n \choose k} \le (\frac{en}{k})^k$,we obtain
\begin{align*}
\P\left(\left|\{j\in \{1, \dots, 2n\}: \Pd(X_j)\neq X_j\} \right|=k \right) &\le
{2n\choose k} \P\left(\Pd(X)\neq X\right)^k\\
&\le  \Big( \frac{e2n}{k}\Big)^k C^k \exp(-ckd^\gamma)\le (Cn\exp(-cd^\gamma))^k.
\end{align*}
Combining these inequalities, we obtain (and enlarging the constant $C$),
\begin{align*}
\E\big[\mathcal{S}_{\mathcal{J}}(X_{1:2n})\big] 
&\le \sum_{k=0}^{2n} 2^{k} (2n-k+1)^{d+1} (Cn\exp(-cd^\gamma))^k\le (2n+1)^{d+1}  \sum_{k=0}^{2n} \left(Cn\exp(-cd^\gamma)\right)^k.
\end{align*}
Choosing $d=(\log(2Cn)/c)^{\frac{1}{\gamma}}$, we have $Cn\exp(-cd^\gamma)=1/2$; thus,
\begin{align*}
\E\big[\mathcal{S}_{\mathcal{J}}(X_{1:2n})\big]  \le (2n+1)^{C(1+\log(n)^{\frac{1}{\gamma}})}.   
\end{align*}
In conclusion, plugging the above bound into \eqref{eq:lem61bound}, we find
\begin{align*}
\P\left(\sup_{(\mathbf{\theta},t)\in \H \times {\mathbb R}}\frac{|F_{\mathbf{\theta}}(t)-F_{\mathbf{\theta},n}(t)|}{\sqrt{F_\theta(t)\vee F_{\theta,n}(t)}}\ge \epsilon \right) \le  Ce^{C\log(2n+1)^{1+\frac{1}{\gamma}}-cn\epsilon^2}.  \end{align*}
\end{proof}

\subsection{Proofs for polynomial decay}

Recall that by Assumption \ref{ass:polyn_decay} have the following:
\begin{align*}
\lambda_j\leq C j^{-\gamma}\qquad \text{ for all }j\geq 1,\text{ where }C>0,\gamma>1.   
\end{align*}

\begin{lemma}\label{lem:pol}
Under Assumption \ref{ass:polyn_decay}, we have $$\int_{\H} \|\Pd (x)-x\|^q_{\H}\,\mu(dx) \leq C d^{-\frac{q\gamma}{2}},$$ for all $d\geq 1$ and all $q\in [1,s]$. The constant $C$ depends on $q$ and $\widetilde{M}_s(\mu)$ only.
\end{lemma}

\begin{proof}
Recall that in the proof of Lemma \ref{lem:exp}, we showed that for $d\ge 1$,
\begin{align*}
\int \|x-\Pd(x)\|^q_{\H} \,\mu(dx) 
&\le \lambda_{d}^{\frac{q}{2}}  \int  \Big( \sum_{j=d+1}^{\infty}\frac{1}{\lambda_j^2}\langle x,\psi_j\rangle^2 \Big)^{\frac{q}{2}} \,\mu(dx).
\end{align*}
As the integral is bounded by $1+\widetilde M_s(\mu)$ and $\lambda_{d}^{\frac{q}{2}}\le C^{\frac{q}{2}} d^{-\frac{q\gamma}{2}}$ by assumption, the claim follows.
\end{proof}

\begin{reptheorem}{thm:conc_poly}
Under Assumptions \ref{ass:mu} and \ref{ass:polyn_decay}, we have 
\begin{align*}
&\P\left( \overline{\mathcal{W}}_p(\mu_n, \mu)^p \ge c\log(2n+1)^{\frac{p}{s}}n^{\frac{1}{2(1+p\gamma)}} \left[\sqrt{\frac{1}{n}} +\left(2+\sqrt{M_{s}(\mu) \vee M_{s}(\mu_n)}\right) \epsilon \right] \right)\\
&\le  Ce^{4n^{\frac{1}{1+p\gamma}}\log(2n+1)-\frac{n\epsilon^2}{64}}+ \frac{C}{\epsilon^{\frac{s}{p}}}n^{-\frac{s\gamma}{2(1+p\gamma)}},
\end{align*}
for all $\epsilon>0$. The constants $c,C$ depend only on $p,s,\gamma, \widetilde{M}_p(\mu)$.
\end{reptheorem}

\begin{proof}
Combining Lemmas \ref{lem:fintoinf} and \ref{lem:pol} with $q=p$ we obtain
\begin{align*}
\overline{\mathcal{W}}_p(\mu^d,\mu)^p \leq C d^{-\frac{p \gamma} {2}}.
\end{align*}
Similarly, using Lemmas \ref{lem:fintoinf}, \ref{lem:fintoinf_prob}, and \ref{lem:pol} with $q=s$ yield
\begin{align*}
\P\left(\overline{\mathcal{W}}_p(\mu^d_n,\mu_n)^p\ge \epsilon\right) \le \frac{C}{\epsilon^{\frac{s}{p}}}d^{-\frac{s\gamma}{2}}.
\end{align*}
Together with \eqref{eq:triangle} and \eqref{eq:finite_d}, we thus conclude
\begin{align*}
&\P\left( \overline{\mathcal{W}}_p(\mu_n, \mu)^p \ge c\log(2n+1)^{\frac{p}{s}} \left[\sqrt{\frac{d}{n}} +\left(2+\sqrt{M_{s}(\mu) \vee M_{s}(\mu_n)}\right) \epsilon \right]+ C d^{-\frac{p \gamma}{2}} \right) \\
&\le  Ce^{2(d+1)\log(2n+1)-\frac{n\epsilon^2}{64}} + \frac{C}{\epsilon^{\frac{s}{p}}}d^{-\frac{s\gamma}{2}}.
\end{align*}
Finally as claimed, choosing $d=n^{\frac{1}{1+p\gamma}}$ yields
\begin{align*}
&\P\left( \overline{\mathcal{W}}_p(\mu_n, \mu)^p \ge c\log(2n+1)^{\frac{p}{s}}n^{\frac{1}{2(1+p\gamma)}} \left[\sqrt{\frac{1}{n}} +\left(2+\sqrt{M_{s}(\mu) \vee M_{s}(\mu_n)}\right) \epsilon \right] \right)\\
&\le  Ce^{4n^{\frac{1}{1+p\gamma}}\log(2n+1)-\frac{n\epsilon^2}{64}}+ \frac{C}{\epsilon^{\frac{s}{p}}}n^{-\frac{s\gamma}{2(1+p\gamma)}}.
\end{align*}

\end{proof}

\begin{reptheorem}{thm:exp_poly}
Under Assumptions \ref{ass:mu} and \ref{ass:polyn_decay}, we have 
\begin{align*}
\E\left[\overline{\mathcal{W}}_p(\mu_n, \mu)^p\right]\le   C \frac{\log(2n+1)^{\frac{p}{s}+\frac12} }{n^{\frac{1}{2}-\frac{1}{1+p\gamma}}} .
\end{align*}
The constant $C$ depends only on $p,s, \gamma, M_s(\mu)$.
\end{reptheorem}

\begin{proof}
We follow the same strategy as in the proofs of Theorem \ref{thm:exp_fin} and \ref{thm:exp_exp}. First, we obtain 
\begin{align*}
&\P\bigg( \overline{\mathcal{W}}_p(\mu_n, \mu)^p \ge c\log(2n+1)^{\frac{p}{s}}n^{\frac{1}{2(1+p\gamma)}} \Big[\sqrt{\frac{1}{n}} +\big(2+\sqrt{2M_s(\mu)+\tilde\epsilon}) \epsilon \Big]\bigg)\\
&\le   Ce^{4n^{\frac{1}{1+p\gamma}}\log(2n+1)-\frac{n\epsilon^2}{64}}+ \frac{C}{\epsilon^{\frac{s}{p}}}n^{-\frac{s\gamma}{2(1+p\gamma)}}+\frac{2M_s(\mu)}{\tilde\epsilon}.   
\end{align*}
Choosing $\epsilon=16v\sqrt{n^{1/(1+p\gamma)}\log(2n+1)/n}$ and $\tilde\epsilon=v^{\frac{s}{p}}$, we obtain
\begin{equation}
\begin{split}
&\P\bigg( \overline{\mathcal{W}}_p(\mu_n, \mu)^p \ge c\frac{\log(2n+1)^{\frac{p}{s}+\frac12} n^{\frac{1}{1+p\gamma}}}{\sqrt{n}} \Big[1+v(1+\sqrt{2M_s(\mu)+v^{\frac{s}{p}}}) \Big]\bigg)\\
&\le  Ce^{4n^{\frac{1}{1+p\gamma}}\log(2n+1)(1-v^2)}+ \frac{C}{v^{\frac{s}{p}}}+ \frac{2M_s(\mu)}{v^{\frac{s}{p}}}.   
\label{eq:expec_bound2}
\end{split}
\end{equation}
Following the same steps as in the proof of Theorem \ref{thm:exp_fin} and \ref{thm:exp_exp}, it remains to evaluate this inequality at $u=v^{1+\frac{s}{2p}}$ in order to bound the $\E[\overline{\mathcal{W}}_p(\mu_n, \mu)^p]$. Noting that $$\Big(1+\frac{s}{2p}\Big)^{-1} \frac{s}{p}=\frac{2s}{2p+s}>1,$$
we conclude that $$\int_1^\infty \left[Ce^{4n^{\frac{1}{p\gamma}}\log(2n+1)(1-u^{4p/(2p+s)})}+ \frac{C}{u^{2s/(2p+s)}}+ \frac{2M_s(\mu)}{u^{2s/(2p+s)}} \right]\,du<\infty. $$
The claim follows.
\end{proof}

Finally, we prove Theorem \ref{thm:unif_poly}. For this, we  recall Assumption \ref{ass:poly2}, which states that for all $d\geq 1$,
\begin{align}\label{eq:ass_exp2}
\mu\Big(\sum_{j=d+1}^\infty \frac{1}{\lambda_j}  \langle x,\psi_j\rangle_{L^2}^2 >0 \Big)\leq Cd^{-\gamma}. \end{align}

\begin{reptheorem}{thm:unif_poly}
Under Assumption \ref{ass:poly2}, we have
\begin{align*}
&\P\Big(\sup_{(\mathbf{\theta},t)\in \H \times {\mathbb R}}\frac{|F_{\mathbf{\theta}}(t)-F_{\mathbf{\theta},n}(t)|}{\sqrt{F_\theta(t)\vee F_{\theta,n}(t)}}\ge \epsilon \Big)\le e^{C \log(2n+1)n^{\frac{1}{\gamma}} -\frac{c n\epsilon^2}{64}},
\end{align*}    
for all $\epsilon>0$, where $c>0$ is an absolute constant and $C>0$ depends on $\gamma$.
\end{reptheorem}

\begin{proof}
We proceed as in the proof of Theorem \ref{thm:unif_exp}.
In particular, by Lemma \ref{lem:ratio}, we have 
\begin{align*}
\P\left(\sup_{(\mathbf{\theta},t)\in \H \times {\mathbb R}}\frac{|F_{\mathbf{\theta}}(t)-F_{\mathbf{\theta},n}(t)|}{\sqrt{F_\theta(t)\vee F_{\theta,n}(t)}}\ge \epsilon \right) \le \E\Big[\mathcal{S}_{\mathcal{J}}(X_{1:2n})\Big] e^{-\frac{c n\epsilon^2}{64}}.
\end{align*}
By the tower property, we have for any $d\ge 1$,
\begin{align*}
\E\Big[\mathcal{S}_{\mathcal{J}}(X_{1:2n})\Big]
&= \sum_{k=0}^{2n} 
\E\Big[\mathcal{S}_{\mathcal{J}}(X_{1:2n})\Big| |\{j\in \{1, \dots, 2n\}:\, \Pd(X_j)\neq X_j\}|=k \Big]\\
&\qquad\qquad \cdot \P(|\{j\in \{1, \dots, 2n\}: \Pd(X_j)\neq X_j\}|=k ).
\end{align*}
On the set $\{|\{j\in \{1, \dots, 2n\}:\, \Pd(X_j)\neq X_j\}|=k\},$ Lemma \ref{lem:ratio} and the definition of $S_{\mathcal{J}}$ imply
\begin{align*}
\mathcal{S}_{\mathcal{J}}(X_{1:2n})\le 2^{k} (2n-k+1)^{d+1},
\end{align*}
and by \eqref{eq:ass_exp2}
\begin{align*}
\P(|\{j\in \{1, \dots, 2n\}: \Pd(X_j)\neq X_j\}|=k ) \le
{2n\choose k} \P(\Pd(X)\neq X)^k&\le  \Big( \frac{e2n}{k}\Big)^k C^k d^{-\gamma k}\\
&\le (Cnd^{-\gamma})^k.
\end{align*}
Combining these inequalities, we obtain
\begin{align*}
\E\Big[\mathcal{S}_{\mathcal{J}}(X_{1:2n})\Big] 
&\le \sum_{k=0}^{2n} 2^{k} (2n-k+1)^{d+1} (Cnd^{-\gamma})^k \le (2n+1)^{d+1}  \sum_{j=0}^{2n} (Cnd^{-\gamma})^k.
\end{align*}
Choosing $d=(2Cn)^{\frac{1}{\gamma}}$, we have $Cnd^{-\gamma} =1/2$; thus,
\begin{align*}
\E\Big[\mathcal{S}_{\mathcal{J}}(X_{1:2n})\Big] \le (2n+1)^{(2Cn)^{\frac{1}{\gamma}}+1}.   
\end{align*}
In conclusion,
\begin{align*}
\P\Big(\sup_{(\mathbf{\theta},t)\in \H \times {\mathbb R}}\frac{|F_{\mathbf{\theta}}(t)-F_{\mathbf{\theta},n}(t)|}{\sqrt{F_\theta(t)\vee F_{\theta,n}(t)}}\ge \epsilon \Big) \le  e^{C \log(2n+1)n^{\frac{1}{\gamma}} -\frac{cn\epsilon^2}{64}}.    
\end{align*}

\end{proof}

\subsection{Proofs for Section \ref{sec:lit_review}}

\begin{lemma}
Under Assumption \ref{ass:lei} with $\tau_j\le C\exp(-cj^\gamma)$ we have $$\int \|\Pd (x)-x\|^q_{L^2}\,\mu(dx) \leq C \exp(-c qd^\gamma)$$ for all $d\geq 1$ and all $q\in [1,s]$. The constant $c$ depends on $\widetilde{M}_s(\mu)$ only, while the constant $C$ depends on $q$ and $\widetilde{M}_s(\mu)$ only.
\end{lemma}

\begin{proof}
We observe that for $d\ge 1$,
\begin{align*}
\int \|x-\Pd(x)\|^q_{L^2} \,\mu(dx) 
=\int  \Big(\sum_{j=d+1}^{\infty} \langle x,\psi_j\rangle^2 \Big)^{\frac{q}{2}} \, \mu(dx) &\leq \sup_{j\geq d+1} \tau_j^{q} \int  \Big( \sum_{j=d+1}^{\infty}\frac{1}{\tau_j^2}\langle x,\psi_j\rangle^2 \Big)^{\frac{q}{2}} \,\mu(dx)\\
&\le \tau_{d}^{q}  \int  \Big( \sum_{j=d+1}^{\infty}\frac{1}{\tau_j^2}\langle x,\psi_j\rangle^2 \Big)^{\frac{q}{2}} \,\mu(dx).
\end{align*}
As the integral is bounded by $1+\widetilde M_s(\mu)$ and $\tau_{d}^{q}\le C^{q} \exp(-cqd^\gamma)$ by assumption, the claim follows.
\end{proof}

\begin{lemma}
Under Assumption \ref{ass:lei} with $\tau_j\le Cj^{-\gamma}$, we have $$\int_{\H} \|\Pd (x)-x\|^q_{\H}\,\mu(dx) \leq C d^{-q\gamma},$$ for all $d\geq 1$ and all $q\in [1,s]$. The constant $C$ depends on $q$ and $\widetilde{M}_s(\mu)$ only.
\end{lemma}

The remainder of the proofs  stays the same, apart from obvious changes in the notation; in consequence, our main results still apply.

\subsection{Proof for Section \ref{sec:numerical}} \label{sec:proof_5}

\begin{lemma}\label{rmk: relations_a_b_c}
In the setting of Proposition \ref{prop: gaussian_kernel_eigenfunctions}, we have the following: 
\begin{align}
a^2+2ab &= c^2, \label{eq:a1} \\
\frac{b(a+c)}{a+b+c} &=c-a,\label{eq:a2} \\
\sqrt{\frac{a+b-c}{a+b+c}}&=\frac{b}{a+b+c}.\label{eq:a3}
\end{align}
\end{lemma}

\begin{proof} 
\eqref{eq:a1} follows directly from the definition in \eqref{eq: def_a_b_c}. For \eqref{eq:a2}, we calculate
\begin{align*}
(a+b+c)(c-a)&=ac+bc+c^2-a^2-ab-ac \stackrel{\eqref{eq:a1}}{=}ab+bc=b(a+c),
\end{align*}
so rearranging yields the claim. Lastly, for \eqref{eq:a3}, we note that 
\begin{equation*}
\sqrt{\frac{a+b-c}{a+b+c}}=\sqrt{\frac{(a+b)^2-c^2}{(a+b+c)^2}}
\stackrel{\eqref{eq:a1}}{=} \frac{b}{a+b+c}.   
\end{equation*}
\end{proof}

\begin{proof}[Proof of Proposition \ref{prop: gaussian_kernel_eigenfunctions}]

We need to show that for any $j\in \N$ and for $ \lambda_j$ and $\psi_j$ given in \eqref{eq: def_psi_j},
\begin{align}\label{eq:TK}
T_K \psi_j(z') =\int  K(z,z')\psi_j (z)\,m(dz)=\lambda_j \psi_j(z'),
\end{align}
where $ K(z,z') = \exp(-(z-z')^2/2w^2)$ is the Gaussian kernel.
To prove the above, we first define 
\begin{align}\label{eq:psi}
\widetilde{\psi}_j(z):= \psi_j(z) \sqrt{ \sqrt{\frac{a}{c}}  2^j j! }   = \exp\left( -(c-a)z^2\right)H_j( \sqrt{2c}z),
\end{align}
where $a, b$, and $c$ are defined in \eqref{eq: def_a_b_c}. We show that 
\begin{equation}\label{eq:toshow}
T_K \widetilde{\psi}_j(z)=  \sqrt{\frac{2a}{a+b+c}} \left( \frac{b}{a+b+c}\right)^j \widetilde{\psi}_j(z) =\lambda_j \widetilde{\psi}_j(z),
\end{equation}
from which \eqref{eq:TK} follows by linearity of $T_K$.

We also recall the following useful identity, stated in  \cite[Equation 7.374.8]{gradshteyn2014table}:
\begin{equation}\label{eq: hermite_eigen}
\int_{\mathbb R}\exp(-(z-z')^2)H_j(\alpha z)\,d z=\sqrt{\pi}(1-\alpha^2)^{\frac{j}{2}}H_j\left(\frac{\alpha z'}{(1-\alpha^2)^{\frac12}}\right).
\end{equation}

Now, using \eqref{eq:psi} as well as the definitions of $a, b$, and $c$  in \eqref{eq: def_a_b_c}, we find
\begin{align*}
\begin{split}
T_K \widetilde{\psi}_j(z') &= \int \exp\left(-\frac{(z-z')^2}{2w^2}\right)\exp\left( -(c-a)z^2\right)H_j( \sqrt{2c}z) \frac{1}{\sqrt{2\pi}\sigma}\exp\left(-\frac{z^2}{2\sigma^2}\right) \,dz \\
&=\int \exp\left(-b(z-z')^2\right)\exp\left( -(c-a)z^2\right)H_j( \sqrt{2c}z)\sqrt{\frac{2a}{\pi}}\exp\left(-2az^2\right) \,dz.
\end{split}
\end{align*}
Next, we notice that by collecting like terms and completing the square we have
\begin{align*}
\begin{split}
&T_K \widetilde{\psi}_j(z') 
=\sqrt{\frac{2a}{\pi}}\int \exp\left(-(b+c+a)z^2+2bzz'-b(z')^2\right)H_j( \sqrt{2c}z)  \,dz\\
&=\sqrt{\frac{2a}{\pi}}\exp\left(-\left(b-\frac{b^2}{a+b+c}\right)(z')^2\right)  \int\exp\left(-\left(\sqrt{a+b+c}\,z-\frac{bz'}{\sqrt{a+b+c}}\right)^2\right)H_j( \sqrt{2c}z) \, dz.
\end{split}
\end{align*}
Noticing that $\exp\left(-\left(b-\frac{b^2}{a+b+c}\right)(z')^2\right)  = \exp\left(-\frac{b(a +c)}{a+b+c}(z')^2\right) $ and from the change of variables $\tilde{z}=\sqrt{a+b+c} \, z$,
\begin{align*}
\begin{split}
T_K \widetilde{\psi}_j(z') 
&=\sqrt{\frac{2a}{\pi(a+b+c)}}\exp\left(-\frac{b(a+c)(z')^2}{a+b+c}\right)\int \exp\left(-\left(\tilde{z}-\frac{bz'}{\sqrt{a+b+c}}\right)^2\right)H_j\left( \sqrt{\frac{2c}{a+b+c}}\tilde{z}\right) \, d\tilde{z}\\
&=\sqrt{\frac{2a}{\pi(a+b+c)}}\exp\left(-\frac{b(a+c)(z')^2}{a+b+c}\right)\left(\sqrt{\pi}\Big(1-\frac{2c}{a+b+c}\Big)^{\frac{j}{2}}H_j\left(\frac{\sqrt{\frac{2c}{a+b+c}} \frac{b}{\sqrt{a+b+c}} z'}{\Big(1-\frac{2c}{a+b+c}\Big)^{\frac{1}{2}}}\right)\right),
\end{split}
\end{align*}
where the final equality uses \eqref{eq: hermite_eigen}.
Finally, simplifying and using Lemma~\ref{rmk: relations_a_b_c}, we show \eqref{eq:toshow}:
\begin{align*}
\begin{split}
T_K \widetilde{\psi}_j(z') 
&=\sqrt{\frac{2a}{\pi(a+b+c)}}\exp\left(-\frac{b(a+c)(z')^2}{a+b+c}\right)\left(\sqrt{\pi}\left(\frac{a+b-c}{a+b+c}\right)^{\frac{j}{2}}H_j\left(\frac{\frac{\sqrt{2c}b}{a+b+c} z'}{\left(\frac{a+b-c}{a+b+c}\right)^{\frac{1}{2}}}\right)\right)  \\
&=\sqrt{\frac{2a}{a+b+c}} \exp(-(c-a)(z')^2)H_j( \sqrt{2c}z')\left( \frac{b}{a+b+c}\right)^j.
\end{split}
\end{align*}

Next, we show that $\{\psi_j\}_{j\in \N}$ are orthonormal. For this, using \eqref{eq:psi} we compute for $j, k\in\N$,
\begin{equation*}
\begin{split}
\int \tilde{\psi}_j(z)\tilde{\psi}_k(z)\frac{1}{\sqrt{2\pi}\sigma}\exp\left(-\frac{z^2}{2\sigma^2}\right) d z &= \int \exp( -2(c-a)z^2)H_j(\sqrt{2c}z )H_k (\sqrt{2c}z ) \sqrt{\frac{2a}{\pi}}\exp(-2az^2) d z\\
&= \sqrt{\frac{2a}{\pi}}\int_{\mathbb R}\exp( -2cz^2)H_j(\sqrt{2c}z)H_k (\sqrt{2c}z )\,d z.
\end{split}
\end{equation*}
Finally, using a change of variables $\tilde z=\sqrt{2c}z$ and \cite[Eq. 7.374.1]{gradshteyn2014table}, we find
\begin{equation*}
\begin{split}
\int &\tilde{\psi}_j(z)\tilde{\psi}_k(z)\frac{1}{\sqrt{2\pi}\sigma}\exp\left(-\frac{z^2}{2\sigma^2}\right)\, d z =\sqrt{\frac{a}{c\pi}}\int_{\mathbb R}\exp\left( -\tilde z^2\right)H_j(\tilde z)H_k(\tilde z)\,d\tilde z =\left\{\begin{array}{cc}
    0,   &  j\neq k,\\
    \sqrt{\frac{a}{c}}  2^j j!,  & j=k,
\end{array}\right.
\end{split}
\end{equation*}
Therefore $\psi_j=\left[ \sqrt{\frac{a}{c}}  2^j j! \right]^{-\frac{1}{2}} \tilde{\psi}_j$ is orthonormal in $L^2(\mathbb R, m )$. By change of variables, completeness of $\{\psi_j\}_{j\in\N}$ in $L^2(\mathbb R, m)$ follows from the well known fact that the  weighted Hermite polynomials form a complete orthonormal basis of $L^2(\mathbb R)$ \cite{weighted_polynomial}. 
\end{proof}

\begin{proof}[Proof of Lemma \ref{lem: rkhs_kernel_embed_eg}]
Recalling \eqref{eq: def_a_b_c} and the definition of $\kappa$, we note that  $c=\sqrt{a^2+2ab}=a\sqrt{1+2\kappa}$. 
Proposition \ref{prop: gaussian_kernel_eigenfunctions}  states that the eigenvalues of $T_K$ are
\begin{equation}\label{eq: eigenvalue_eg}
\lambda_j= \sqrt{\frac{2a}{a+b+c}} \left( \frac{b}{a+b+c}\right)^j= \sqrt{\frac{2}{1+\kappa+\sqrt{1+2\kappa}}} \left( \frac{\kappa}{1+\kappa+\sqrt{1+2\kappa}}\right)^j.
\end{equation}
As $\kappa\geq 4$, it can be checked that
\begin{equation*}
\frac{1}{2}\le \frac{\kappa}{1+\kappa+\sqrt{1+2\kappa}}<1,
\end{equation*}
which implies that
\begin{equation}
\label{eq:lambda_bound}
\sqrt{\frac{2}{1+\kappa+\sqrt{1+2\kappa}}}  \left(\frac{1}{2}\right)^j \le \lambda_j < \sqrt{\frac{2}{1+\kappa+\sqrt{1+2\kappa}}} \leq \frac{1}{2}.
\end{equation}
In particular, Assumption \ref{ass:exponential decay}, namely, $\lambda_j \leq C \exp(-c j^{\gamma})$ is satisfied for $\gamma=1$ using \eqref{eq:lambda_bound}. 

Now we show (2), i.e.\ $\widetilde{M}_s(\mu)<\infty$. First, recalling that $\tilde \mu=\mathcal{N}(0,\eta^2)$ and $\mu=\phi_{\#}\tilde \mu$, we have
\begin{align*}
\widetilde{M}_s(\mu) =\int \Big(\sum_{j=1}^{\infty}\frac{1}{\lambda_j^2}\langle x, \psi_j\rangle^2_{L^2} \Big)^{\frac{s}{2}} \mu(dx)=\int  \Big(\sum\limits_{j=1}^{\infty}\frac{|\psi_j(z)|^2}{\lambda_j} \Big)^{\frac{s}{2}} \tilde \mu(dz) &=\int \Big(\sum\limits_{j=1}^{\infty}\lambda_j\frac{|\psi_j(z)|^2}{\lambda^2_j}\Big)^{\frac{s}{2}} \tilde \mu(dz)\\
&\overset{\text{Jensen}}{\leq} C\int \Big(\sum\limits_{j=1}^{\infty}\lambda_j\frac{|\psi_j(z)|^s}{\lambda^s_j}\Big) \tilde \mu(dz).
\end{align*}
Next, using Fubini Theorem, we find
\begin{align*}
\widetilde{M}_s(\mu) \leq C\int \Big(\sum\limits_{j=1}^{\infty}\lambda_j\frac{|\psi_j(z)|^s}{\lambda^s_j}\Big)\tilde \mu(dz)  & =C\int \Big(\sum\limits_{j=1}^{\infty}\frac{|\psi_j(z)|^s}{\lambda_j^{\frac{s}{2}}} \Big(\frac{1}{\lambda_j}\Big)^{\frac{s}{2}-1}\Big)\tilde \mu(dz)\\
&\overset{\text{Fubini}}{=}C\sum\limits_{j=1}^{\infty}\int \frac{|\psi_j(z)|^s}{\lambda_j^{s/2}} \Big(\frac{1}{\lambda_j}\Big)^{\frac{s}{2}-1}\tilde \mu(dz).
\end{align*}

Notice that by \eqref{eq:lambda_bound}, for a constant $C>0$ depending on $a,b$, we have
\begin{align}\label{eq: kappa_eg}
\frac{1}{\sqrt{\lambda_j}}\le C 2^{j/2};
\end{align}
therefore,
\begin{align*}
&\widetilde{M}_s(\mu) \leq C\sum\limits_{j=1}^{\infty}\int \Big(\frac{|\psi_j(z)|^s}{\lambda_j^{s/2}}\Big) \Big(\frac{1}{\lambda_j}\Big)^{\frac{s}{2}-1}\,\tilde \mu(dz)\\
&\overset{\eqref{eq: def_psi_j},\eqref{eq: kappa_eg}}{\leq}  C\sum\limits_{j=1}^{\infty}\int \Big( \frac{1}{   j! }\Big)^{\frac{s}{2}}\Big(\frac{1}{\lambda_j}\Big)^{\frac{s}{2}-1}\Big(\exp(sa z^2)  \|H_j(\sqrt{2c z})\exp(-c z^2)\|^s_{L^{\infty}(\mathbb R)}\Big) \tilde \mu(dz)\\
&\overset{\eqref{eq: pointwise_eigenfunction_upperbound}}{\leq} C\sum\limits_{j=1}^{\infty}\int \Big( \frac{1}{ j! }\Big)^{\frac{s}{2}}\Big(\frac{1}{\lambda_j}\Big)^{\frac{s}{2}-1}\Big(\exp(sa z^2)  j^{-\frac{s}{12}}\Big)\frac{1}{\sqrt{2\pi}\eta}\exp\Big(-\frac{z^2}{2\eta^2}\Big) dz,
\end{align*}
where the final step uses that by \cite[Theorem 1]{BONAN1990210}, for each $j\in \N$, we have
\begin{equation}\label{eq: pointwise_eigenfunction_upperbound}
\max\limits_{x\in\mathbb R}\left|H_j(\sqrt{2c}x)\right|\exp(-cx^2)\leq D j^{-\frac{1}{12}},
\end{equation}
where $D$ is a positive constant.  Finally, by assumption, we have
\begin{equation}\label{eq: sa-1/2}
sa-\frac{1}{2\eta^2}=a\Big(s- \frac{2\sigma^2}{\eta^2}\Big)<0.
\end{equation}
Therefore, 
\begin{align*}
&\widetilde{M}_s(\mu)
\leq C\sum\limits_{j=1}^{\infty}\int \Big( \frac{1}{ j! }\Big)^{\frac{s}{2}}\Big(\frac{1}{\lambda_j}\Big)^{\frac{s}{2}-1}\left(\exp(sa z^2)  j^{-\frac{s}{12}}\right)\frac{1}{\sqrt{2\pi}\eta}\exp\Big(-\frac{z^2}{2\eta^2}\Big) dz\\
&\leq C\Big[\int\exp\Big(\Big(sa-\frac1{2\eta^2}\Big) z^2\Big)\frac{1}{\sqrt{2\pi}\eta}dz\Big]\sum\limits_{j=1}^{\infty}\Big( \frac{1}{ j! }\Big)^{\frac{s}{2}}\Big(\frac{1}{\lambda_j}\Big)^{\frac{s}{2}-1} j^{-\frac{s}{12}} \overset{\eqref{eq: sa-1/2}}{\leq}C\sum\limits_{j=1}^{\infty}\Big( \frac{1}{ j! }\Big)^{\frac{s}{2}}\Big(\frac{1}{\lambda_j}\Big)^{\frac{s}{2}-1} j^{-\frac{s}{12}}.
\end{align*}
The proof is completed by noting that
\begin{align*}
\widetilde{M}_s(\mu)
\leq C\sum\limits_{j=1}^{\infty}\Big( \frac{1}{  j! }\Big)^{\frac{s}{2}}\Big(\frac{1}{\lambda_j}\Big)^{\frac{s}{2}-1} j^{-\frac{s}{12}} &\leq C\sum\limits_{j=1}^{\infty}\Big( \frac{1}{ j! }\Big)^{\frac{s}{2}}(2^j)^{\frac{s}{2}-1}  j^{-\frac{s}{12}}=C\sum\limits_{j=1}^{\infty}\Big( \sqrt{\frac{2^j}{j!}}\Big)^{s}2^{-j}  j^{-\frac{s}{12}}<\infty,
\end{align*}
where in the last inequality, convergence of sum is guaranteed by Stirling's approximation. The constant $C$ only depends on $a,b$ and $s$.
\end{proof}

\section{Some notes of the topology of generated by $ \overline{\mathcal{W}}_p$} \label{sec:topology}

In this section we give some details on the topology induced by the MSW distance.
First, recall the definition of weak convergence. 

\begin{definition}
We say that a sequence of probability measures $(\mu^n)_{n\in \N}$ on $\mathcal{H}$ converges weakly to a probability measure $\mu$ if for every bounded continuous $f:\mathcal{H}\to \R$,
\[
\lim\limits_{n\to\infty}\int_{X}f(x)\,\mu^n(d x)=\int_{X}f(x)\,\mu(dx).
\]

\end{definition}

We start with an easy lemma:
\begin{lemma}\label{est: SW<W}
If $\mu,\nu$ have finite $p$th moment, then $\overline{\mathcal{W}}_{p}(\mu,\nu)\leq \mathcal{W}_p(\mu,\nu)$.
\end{lemma}

\begin{proof}
 By e.g.\ \cite[Theorem 1.7]{santambrogio2015optimal}, there exists an optimal transport plan $\pi\in \Pi(\mu,\nu)$ for $W_p(\mu,\nu)$. Then $(\text{Proj}_{\theta}\times \text{Proj}_{\theta})_{\#} \pi\in \Pi(\mu_{\theta},\nu_{\theta})$. Then, by Cauchy-Schwarz with $\|\theta\|_{\mathcal{H}}=1$, we have
\[
\mathcal{W}_p(\mu_\theta,\nu_\theta)^p\leq \int |\langle \theta,x\rangle-\langle \theta,y\rangle|^p\,\pi(dx,dy)\leq \int \|x-y\|_{\H}^p\,\pi(dx,dy)=\mathcal{W}_p(\mu,\nu)^p.
\]
Taking the supremum over $\|\theta\|_\mathcal{H}=1$ concludes the proof. 
\end{proof}

The following lemma states the equivalence between weak convergence and convergence with respec to $\overline{\mathcal{W}}_p$.  

\begin{lemma}\label{sw: weak topo_w}
Let $\mu^n,\mu$ have finite $p$th moment. 
\begin{enumerate}
    \item If  $\lim_{n\to\infty} \overline{\mathcal{W}}_p(\mu^n,\mu)=0$, then $(\mu^n)_{n\in \N}$ converges weakly to $\mu$. 
    \item If $(\mu^n)_{n\in \N}$ converges weakly to $\mu$ and $\lim_{n\to\infty}\int \|x\|_{\H}^p\,\mu^n(d x)=\int_{\H}\|x\|_{\H}^p\,\mu(dx)$, then $$\lim_{n\to \infty}\overline{\mathcal{W}}_p(\mu^n,\mu)= 0.$$
\end{enumerate}
\end{lemma}

Interestingly, \cite[Example 3.7]{han2023sliced} shows that $\lim_{n\to \infty} \overline{\mathcal{W}}_p(\mu^n,\mu)=0$ does \emph{not} imply $$\lim_{n\to \infty} \int \|x\|_\mathcal{H}^p\,\mu^n(dx) =\int \|x\|_\mathcal{H}^p\,\mu(dx),$$ in general, if $\mathcal{H}$ is infinite dimensional.

\begin{proof}
For (1), we note that $\lim_{n\to\infty}\overline{\mathcal{W}}_p(\mu^n,\mu)=0$ implies that $\lim_{n\to \infty} \mathcal{W}_p(\mu^{n}_\theta, \mu_\theta)=0$ for every $\theta$ suh that $\|\theta\|_{\H}=1$. Thus, for every $\|\theta\|_{\H}=1$, $(\mu^n_\theta)_{n\in \N}$ converges weakly to $\mu_\theta$. Now for every $\gamma \in \H$ and every $a\in \R$,
\begin{equation*}
\begin{split}
\lim_{n\to\infty}\int \exp(a\langle \gamma,x \rangle) \,\mu^{n}(dx)=\lim_{n\to\infty}\int \exp\left(a\|\gamma\|\langle \frac{\gamma}{\|\gamma\|},x \rangle\right) \,\mu^{n}(dx)&=\lim_{n\to\infty}\int \exp(a\|\gamma\|x)\,\mu^{n}_{\frac{\gamma}{\|\gamma\|}}(d x)\\
&=\int \exp(a\|\gamma \|x)\,\mu_{\frac{\gamma}{\|\gamma\|}}(dx)\\
&=\int \exp(a\langle \gamma,x \rangle)\, \mu(d x).
\end{split}
\end{equation*}
This implies that $(\mu^n)_{n\in \N}$ converges weakly to $\mu$ by \cite[ Proposition 4.6.9]{bogachev_weak_convergence_measure}, as claimed.

For (2), we note that $\lim_{n\to \infty} \mathcal{W}_p(\mu^n,\mu)= 0$ by \cite[Definition 6.8 and Theorem 6.9]{villani2009optimal}. Thus, using Lemma \ref{est: SW<W}, $\lim_{n\to \infty} \overline{\mathcal{W}}_p(\mu^n,\mu)= 0$.
\end{proof}

\section{Conclusion}
In this work, we show that the max-sliced Wasserstein distance does not suffer from the curse of dimensionality under quite general conditions when considering its application to probability measures in Hilbert space. There are many open questions that remain. It is natural to ask whether the rates found in this work could be improved with stronger assumptions on the measure.  Furthermore, rates of concentration in Hilbert space for other variants of Wasserstein metrics, like the Gromov-Wasserstein metric, have not been studied to our knowledge.

\bibliographystyle{siam}
\bibliography{bib}

\end{document}